\documentclass[sn-chicago,Numbered]{sn-jnl}

\usepackage{graphicx}%
\usepackage{multirow}%
\usepackage{amsmath,amssymb,amsfonts}%
\usepackage{amsthm}%
\usepackage{mathrsfs}%
\usepackage[title]{appendix}%
\usepackage{xcolor}%
\usepackage{textcomp}%
\usepackage{manyfoot}%
\usepackage{booktabs}%
\usepackage[section]{algorithm}
\usepackage{algorithmicx}%
\usepackage{algpseudocode}%
\usepackage{listings}%

\usepackage{enumerate}
\usepackage{epsfig}
\usepackage{verbatim}
\usepackage{bm} 

\algnewcommand\algorithmicinput{\textbf{Input:}}
\algnewcommand\Input{\item[\algorithmicinput]}
\algnewcommand\algorithmicoutput{\textbf{Output:}}
\algnewcommand\Output{\item[\algorithmicoutput]}

\theoremstyle{thmstyleone}%
\newtheorem{theorem}{Theorem}[section]
\newtheorem{proposition}[theorem]{Proposition}%
\newtheorem{lemma}[theorem]{Lemma}%
\newtheorem{corollary}[theorem]{Corollary}%
\newtheorem{remark}[theorem]{Remark}%
\newtheorem{definition}[theorem]{Definition}%

\numberwithin{equation}{section}

\providecommand{\dd}{\ensuremath{\mathrm{d}}}
\providecommand{\e}{\varepsilon}

\providecommand{\R}{\ensuremath{\mathbb{R}}}
\providecommand{\N}{\ensuremath{\mathbb{N}}}

\newcommand{\Cheb}{Cheby\v{s}ev }
\newcommand{\realiz}{\mathrm{R}}
\newcommand{\spike}{\mathrm{S}}

\newcommand{\depth}{L}
\newcommand{\size}{M}

\newcommand{\Parallel}[1]{{\rm P}(#1)} 
\newcommand{\FParallel}[1]{{\rm FP}(#1)}

\DeclareMathOperator{\Id}{Id}

\newcommand{\ch}{\mathrm{Cheb}}

\newcommand{\norm}[2][]{\| #2 \|_{#1}} 
\newcommand{\snorm}[2][]{| #2 |_{#1}} 
\newcommand{\normc}[2][]{\left\| #2 \right\|_{#1}} 
\newcommand{\snormc}[2][]{\left| #2 \right|_{#1}} 
\newcommand{\ceil}[1]{\lceil #1 \rceil}

\newcommand{\floor}[1]{\lfloor #1 \rfloor}

\newcommand{\idnna}[4]{\Phi^{\Id}_{#1,#2,#3,#4}} 
\newcommand{\prodlabel}{\operatorname{Prod}}
\newcommand{\prodnna}[3]{\Phi^{\prodlabel}_{#2,#3}} 

\begin{document}

\title{Neural Networks for Singular Perturbations}


\author[1]{\fnm{J. A. A.} \sur{Opschoor}}\email{joost.opschoor@sam.math.ethz.ch}
\author*[1]{\fnm{Ch.} \sur{Schwab}}\email{christoph.schwab@sam.math.ethz.ch}
\author[2]{\fnm{C.} \sur{Xenophontos}}\email{xenophontos.christos@ucy.ac.cy}


\affil[1]{\orgdiv{Seminar for Applied Mathematics}, \orgname{ETH Z\"{u}rich}, \orgaddress{\street{R\"{a}mistrasse 101}, \city{Z\"urich}, \postcode{8092}, \country{Switzerland}}} 

\affil[2]{\orgdiv{Department of Mathematics and Statistics}, \orgname{University of Cyprus}, \orgaddress{\street{P.O.~BOX 20537}, \city{Nicosia}, \postcode{1678}, \country{Cyprus}}} 

\abstract{We prove deep neural network (DNN for short) 
expressivity rate bounds for
solution sets of
a model class of singularly perturbed, elliptic two-point
boundary value problems, in Sobolev norms, on the bounded interval $(-1,1)$. 
We assume that the given source term and reaction
coefficient are analytic in $[-1,1]$.

We establish expression rate bounds in Sobolev norms in terms
of the NN size which are uniform
with respect to the singular perturbation parameter
for several classes of DNN architectures.
In particular,
ReLU NNs, spiking NNs, and $\tanh$- and sigmoid-activated NNs.
The latter activations can represent 
``exponential boundary layer solution features'' 
explicitly, in the last hidden layer of the DNN, 
i.e. in a shallow subnetwork, and afford
improved robust expression rate bounds in terms of the NN size.

We prove that all DNN architectures allow 
\emph{robust exponential solution expression} 
in so-called `energy' as well as in `balanced' Sobolev norms,
for analytic input data.}



\maketitle

\section{Introduction}
\label{sec:intro} 
Singular perturbations are ubiquitous in engineering
and in the sciences.
Let us mention only solid mechanics (theory of thin solids,
such as beams, plates and shells), fluid mechanics 
(viscous flows at large Reynolds number), and electromagnetics
(eddy current problems in lossy media).
In all these applications, PDEs depend on a small 
parameter $\e \in (0,1]$ in physically relevant regimes of input data.
Standard numerical approximation methods (Finite Volume, 
Finite Difference or Finite Element) generally do \emph{not}
perform uniformly w.r. to the physical perturbation parameter:
in general, a so-called \emph{scale resolution} condition
relating the discretization parameters (such as the meshwidth
$h$ in Finite Element Methods) and $\e$ needs to hold.

A common trait of singularly perturbed, \emph{elliptic} 
PDE problems is an additive decomposition
of PDE solutions into a regular (typically analytic)
part and into singular components. 
See, e.g., \cite{TemamSgPert} and the references there.
The regular solution part $u^S_\e$ 
may depend on the singular perturbation parameter $\e$, 
but derivatives of the smooth part $u^S_\e$
satisfy bounds in Sobolev norms which are uniform in terms of $\e$. 
We refer to \cite{TemamSgPert} and the references there for 
examples.

The singular perturbation part $u^{BL}_\e$, on the contrary, 
is not uniformly smooth in terms of $\e$. 
Its $k$-th derivative typically grows as $O(\e^{-k})$.
Due to their exponential decay with respect to the distance
to the boundary, they are referred to as \emph{boundary layers}, 
and denoted herein by $u^{BL}_\e$.
\subsection{Contributions}
\label{sec:Contrib}
We prove that exponential boundary layer functions $u^{BL}_\e$
which arise in asymptotic expansions of singularly perturbed
elliptic boundary value problems can be
expressed at exponential (w.r. to the NN size) rates
by various DNN architectures \emph{robustly}, i.e.
\emph{uniformly with respect to the perturbation parameter}, 
in various Sobolev norms. 
The DNNs considered are strict ReLU NNs 
(Propositions~\ref{prop:relubalanced} and \ref{prop:reluexp}),
spiking NNs 
(Theorem~\ref{thm:spikingbalanced}),
and 
$\tanh$-activated NNs (Theorem~\ref{thm:nntanhsingpert}).
\subsection{Layout}
\label{sec:Layout}
Section~\ref{sec:model} introduces 
a model singularly perturbed reaction-diffusion
two-point boundary value problem, 
specifies the assumptions on 
the problem data and recaps several results on the analytic regularity
and the asymptotic behavior of its solutions.

Section~\ref{sec:approx} addresses mostly known FE approximation 
results, in particular featuring so-called 
\emph{robust exponential convergence rates} 
for the parametric solutions.

Section~\ref{sec:nn} 
introduces assumptions on the architecture of the NNs.

Section~\ref{sec:relunn} shows robust exponential convergence 
of strict ReLU NNs, i.e. neural networks with 
only ReLU activations: 
Proposition~\ref{prop:relubalanced} 
states that emulation accuracy $\tau>0$
in a so-called ``balanced norm'' can be achieved 
uniformly w.r. to the perturbation parameter
$0< \e \leq 1$ 
with a ReLU NN of 
size $O(|\log(\tau)|^2)$ and of
depth $O(|\log(\tau)|(1+|\log(|\log(\tau)|))$, i.e.
the constants which are implicit in $O(\; )$ are
independent of $\e$.

Section~\ref{sec:spiking}
establishes a corresponding conclusion for spiking NNs
using a ReLU NN-to-spiking NN conversion algorithm 
from \cite{SWBCPG2022}.

Assuming a constant reaction coefficient function,
stronger robust exponential solution expression rate bounds 
by strict $\tanh$ NNs are proved in Section \ref{sec:tanhanal}.
We prove in Theorem~\ref{thm:nntanhsingpert} 
that to achieve expression error $\tau > 0$ in balanced norms,
$\tanh$-activated NNs of  
depth $O(\log|\log(\tau)|)$ 
and 
size $O(|\log(\tau)|)$ are sufficient.
A corresponding result for the so-called sigmoid activation
is shown in Appendix \ref{sec:GenAct}.

Appendix~\ref{sec:nntanhunivarcheb} contains statement and
proof of expression rate bounds for \Cheb polynomials by 
strict $\tanh$ NNs which are used in various places
and are of independent interest.

The generalization of this \Cheb polynomial emulation
to general smooth activation functions 
is established in Appendix \ref{sec:GenAct}.
\section{The Model Problem and its Regularity}
\label{sec:model}
Consider the following linear,
singularly perturbed, reaction-diffusion 
boundary value problem (BVP): 
find $u_{\e}(x)$ such that
\begin{eqnarray}
-\e^2 u''_{\e}(x) + b(x)u_{\e}(x) &=& f(x) \; , \; x \in I = (-1,1) \; , \label{eq:de} \\
u_{\e}(\pm 1) &=& 0, \label{eq:bc}
\end{eqnarray}
where $\e \in (0,1]$ is a small parameter that can approach zero, and $b(x), f(x)$ are given analytic functions on
$\overline{I}=[-1,1]$, with $b(x) \ge \underline{b} > 0$ on $\overline{I}$ for some constant $\underline{b}$. 
Moreover, we assume there exist positive constants $C_f, K_f, C_b, K_b$ such that $\forall \; n \in \N_0$, there holds
\begin{equation}
\label{eq:analyticfb}
\Vert f^{(n)} \Vert_{L^{\infty}(I)} \leq C_f K^n_f n! \: , \: \Vert b^{(n)} \Vert_{L^{\infty}(I)} \leq C_b K^n_b n! .
\end{equation}
The above problem was studied in \cite{Melenk1997}
where the following result was established.
\begin{theorem} [{\cite[Thm.~{1}]{Melenk1997}}]
\label{thm:analregul}
For $0 < \e \leq 1$, there exists a unique solution $u_{\e}\in H^1_0(I)$ 
of (\ref{eq:de})--(\ref{eq:bc}). 
There exist positive constants $C, K$, independent of $\e$, such that
\begin{equation}
\label{eq:analyticu}
\Vert u_{\e}^{(n)} \Vert_{L^2(I)} \leq C K^n \max\{n, \e^{-1} \}^n \; \forall \; n \in \N_0.
\end{equation}
\end{theorem}
The above corresponds to classical differentiability and it is useful in the case when $\e$ is large. If one
uses the method of matched asymptotic expansions, a more refined regularity result can be obtained, as stated 
below.
\begin{proposition}[{\cite{Melenk1997}}]
\label{prop:decomp}
Let  $u_{\e} \in H^1_0(I)$ be the solution of (\ref{eq:de})--(\ref{eq:bc}) and assume (\ref{eq:analyticfb}) holds.
Then,  $u_{\e}$ may be decomposed as
\begin{equation}
\label{eq:decomp}
u_{\e} 
	= u^S_{\e} + u^{BL}_{\e} + u^R_{\e}
	= u^S_{\e} + u^{+}_{\e} + u^{-}_{\e} + u^R_{\e},
\end{equation}
where $u^S_{\e}$ denotes the smooth part, $u^{\pm}_{\e}$ denote the boundary layers at the two endpoints,
and $u^R_{\e}$ denotes the remainder. 
Furthermore, 
there exist positive constants $C_1, K_1, C_2, K_2, C_3, K_3$ independent of $\e$, 
such that 
\begin{align}
\left\Vert \left( u^S_{\e} \right) ^{(n)} \right\Vert_{L^2(I)} 
	\leq &\, C_1 K_1^n n! ,
\qquad
\text{ for all } n \in \N_0,
\label{eq:smooth} 
\\
\left\vert \left( u^{\pm}_{\e} \right)^{(n)}(x) \right\vert 
	\leq &\, C_2 K_2^n e^{-\sqrt{\underline{b}} (1 \mp x) /\e }
	\max\{n, \e^{-1} \}^n,
\qquad
\forall x \in \overline{I} , n \in \N_0,
\label{eq:BL} 
\\
\left\Vert \left( u^R_{\e} \right) ^{(n)} \right\Vert_{L^2(I)} 
	\leq &\, C_3 \e^{2-n} e^{-K_3 / \e},
\qquad
\text{ for all } 
n \in \{0, 1, 2\}.
\label{eq:rem}
\end{align}
\end{proposition}
\begin{proof}
This follows from the results in Section 2 of \cite{Melenk1997}.
\end{proof}

As can be seen from \eqref{eq:analyticu},
the norms of the derivatives of $u_{\e}$ may grow when $\e\to 0$.
For $u_{\e}$ and its first derivative,
a more precise estimate of this $\e$-dependence 
is stated in the following lemma.
\begin{lemma}[{\cite{MX2016}}]
\label{lem:normsue}
Let  $u_{\e}$ be the solution $u_{\e}$ of (\ref{eq:de})--(\ref{eq:bc}) and assume (\ref{eq:analyticfb}) holds.
Then, there exists a constant $C>0$, independent of $\e$,
such that
\begin{align*}
\Vert u_{\e} \Vert_{L^2(I)}
	\leq C
	,
	\qquad
\Vert u_{\e}' \Vert_{L^2(I)}
	\leq C \e^{-1/2}
	,
	\qquad
\Vert u_{\e} \Vert_{L^\infty(I)}
	\leq C.
\end{align*}
\end{lemma}

\begin{proof}
From Proposition \ref{prop:decomp}, 
it follows that 
there exist constants $C_1,K_1>0$ independent of $\e$ such that,
for every $0<\e\leq 1$ and for every $n\in \N_0$
holds
\begin{align*}
\norm[L^2(I)]{ (u^S_{\e})^{(n)} } \leq C_1 K_1^{n} n! 
\;.
\end{align*}
This can be combined with the interpolation inequality\footnote{Follows from Cauchy-Schwarz and
$(u(x))^2 - (u(y))^2 
= 
\int_y^x (u^2)'(\xi)d\xi$ 
for $-1\leq y<x \leq 1$ and $u\in H^1(I)$.}
\begin{align*}
\norm[L^\infty(I)]{ u^S_{\e} } 
	\leq C \norm[L^2(I)]{ u^S_{\e} }^{1/2} \norm[H^1(I)]{ u^S_{\e} }^{1/2}
\end{align*}
to obtain
$\norm[L^\infty(I)]{ u^S_{\e} } \leq C$,
for some $C>0$ independent of $\e$.

Similarly, 
$\norm[L^2(I)]{ u^R_{\e} } \leq C_3$ and
$\norm[L^2(I)]{ (u^R_{\e})' } \leq C_3$
imply that
$\norm[L^\infty(I)]{ u^R_{\e} } \leq C$,
for $C>0$ independent of $\e$.

Finally,
for the boundary layers we use
\cite[Equation (2.19)]{MX2016},
which is a sharper bound in terms of $\e$ 
than \eqref{eq:BL} in Proposition \ref{prop:decomp}.
It states that
\begin{align}
\label{eq:blnorms}
\e^{-1/2} \Vert u^{\pm}_{\e} \Vert_{L^2(I)}
+ \e^{1/2} \Vert (u^{\pm}_{\e})' \Vert_{L^2(I)}
+ \Vert u^{\pm}_{\e} \Vert_{L^\infty(I)}
	\leq C
	.
\end{align}

Combining these estimates for the terms in \eqref{eq:decomp}
finishes the proof.
\end{proof}

\begin{remark}
\label{rem:blexp}
In the case of constant coefficients, i.e. $b(x) = b \in \R$, $b > 0$ in (\ref{eq:de}), 
the boundary layer parts of the solution may
be explicitly obtained, 
as was the case in \cite[Theorem 2.1]{SS1996}.
Denote again by $u_{\e}$ the solution of (\ref{eq:de})--(\ref{eq:bc}) and assume (\ref{eq:analyticfb}) holds.

Then,  $u_{\e}$ may be decomposed as
\begin{equation}
\label{eq:decompexplicit}
\tilde{u}_{\e} = u^S_{\e} + \tilde{u}^{+}_{\e} + \tilde{u}^{-}_{\e} + u^R_{\e},
\end{equation}
where $u^S_{\e}$ and $u^R_{\e}$ denote the smooth part and the remainder 
from Proposition \ref{prop:decomp}
and
\begin{align}
\label{eq:explicitbl}
\tilde{u}^{\pm}_{\e}(x) = C^{\pm} e^{- \sqrt{b} (1\mp x)/\e },
\end{align}
where the constants $C^{\pm}$ are bounded independently of $\e$ 
(see \cite{SS1996} for more details).
These boundary layer functions 
are related to $u^{\pm}_{\e}$ from Proposition \ref{prop:decomp}
through
$u^{+}_{\e} + u^{-}_{\e} = u^{BL}_{\e} = \tilde{u}^{+}_{\e} + \tilde{u}^{-}_{\e}$.
\end{remark}
\section{$hp$-Approximation}
\label{sec:approx}
The approximation to the solution of (\ref{eq:de})--(\ref{eq:bc}) 
by the Finite Element Method (FEM) was studied in
\cite{Melenk1997}, and in \cite{SS1996}.
In this section we summarize the relevant results.
First, we cast \eqref{eq:de} -- \eqref{eq:bc}
into an equivalent weak formulation that reads: 
find $u_{\e} \in H^1_0(I)$ such that for all $v \in H^1_0(I)$
there holds
\begin{equation}
\label{eq:variational}
\int_0^1 \left\{ \e^2 u'_{\e} v' + b u_{\e} v \right\} dx = \int_0^1 f v dx.
\end{equation}
In order to define the discrete version of (\ref{eq:variational}), 
for $N\in\N$,
let $\Delta = \{ x_j \}_{j=0}^{N}$ be an arbitrary
partition of $I$ and set $I_j = (x_{j-1} , x_{j} )$, $h_j = x_{j} - x_{j-1}, j=1, \ldots, N$. 
With $\mathcal{P}_p(I)$ the space of polynomials of degree at most $p$ on $I$, 
we define the spaces
\begin{eqnarray}
\label{eq:Sp}
\mathcal{S}^p(\Delta) &:=& \{ w \in H^1(I) : w|_{I_j} \in \mathcal{P}_p(I_j), j=1,\dots,N \} ,
\\
\mathcal{S}^p_0(\Delta) &:=& \mathcal{S}^p(\Delta) \cap H^1_0(I) \label{eq:Sp0}.
\end{eqnarray}
The discrete version of (\ref{eq:variational}), then reads: find $u^{FEM}_{\e} \in \mathcal{S}_0^p(\Delta)$ such that
for all $v\in \mathcal{S}_0^p(\Delta)$, there holds
\begin{equation}
\label{eq:discrete}
\int_0^1 \left\{ \e^2 \left(u^{FEM}_{\e}\right)' v' + b u^{FEM}_{\e} v \right\} dx = \int_0^1 f v dx.
\end{equation}
Associated with the above problem, we have the so-called \emph{energy norm}:
\begin{equation}
\label{eq:energy}
\Vert w \Vert_{\e}^2 := \int_0^1 \left\{\e^2 (w')^2 + b w^2\right\} dx, \; w \in H^1_0(I),
\end{equation}
and the usual best approximation property holds:
\begin{equation}
\label{eq:best_approx}
\Vert u_{\e} - u^{FEM}_{\e} \Vert_{\e} \leq \Vert u_{\e} -v \Vert_{\e} \; \forall \; v \in \mathcal{S}_0^p(\Delta).
\end{equation}

The following \emph{spectral boundary layer mesh} 
is the minimal one which yields 
\emph{exponential} convergence rates in terms of the number of degrees of freedom (i.e. the 
number of ``Finite-Element features'')
as the polynomial degree $p$ is increased.
\begin{definition}[Spectral Boundary Layer mesh, {\cite[Definitions 13 and 14]{Melenk1997}}]
\label{def:SBL}
For $\kappa>0$, $p\in \N$ and $0<\e \leq 1$, 
the \emph{Spectral Boundary Layer mesh} 
$\Delta_{BL}(\kappa,p)$ is defined as 
$$
\Delta_{BL}(\kappa,p):= 
\begin{cases}
\{-1,-1+\kappa p \varepsilon,1-\kappa p \varepsilon,1\} & \mbox{ if $\kappa p \varepsilon < 1/2$} \\
\{-1,1\} & \mbox{ if $\kappa p \varepsilon \ge 1/2$}.  
\end{cases}
$$
Furthermore, let us define the spaces 
$V^p(\kappa)$ and $V^p_0(\kappa)$ of piecewise polynomials of degree at most $p$
via
\begin{eqnarray*}
V^p(\kappa):= {\mathcal S}^p(\Delta_{BL}(\kappa,p)),
\;
V^p_0(\kappa) := {\mathcal S}^p_0(\Delta_{BL}(\kappa,p)) = V^p(\kappa) \cap H^1_0(I). 
\end{eqnarray*}
\end{definition}
Using the above mesh, the following was shown in \cite{Melenk1997}.
\begin{proposition}[{\cite[Thm.~{16}]{Melenk1997}}]
\label{prop:Melenk1997} 
Assume that (\ref{eq:analyticfb}) holds and 
let $u_{\e}$ be the solution of (\ref{eq:variational}). 
Then, there exists $\kappa_0>0$ (depending only on $b$ and $f$) such that 
for every $\kappa \in (0,\kappa_0)$ and $p\in\N$
there exist positive constants $C$, $\beta$, independent of $\e$ and $p$, such that 
\begin{equation}
\label{eq:Melenk1997} 
\inf_{v \in V^p_0(\kappa)} \|u_{\e} - v\|_{\e} \leq C e^{-\beta p}. 
\end{equation}
\end{proposition}
For the ensuing deep NN approximation constructions, 
it is important to note that the proof of the above result is constructive, 
in that $v \in V^p_0(\kappa)$ can be
taken to be the element-wise 
Gau\ss-Lobatto interpolant of $u_{\e}$. 
Hence, 
knowledge of the values of $u_{\e}$ in the Gau\ss-Lobatto points in 
each (sub)interval of $\Delta_{BL}(\kappa, p)$ 
is the only required information for constructing $v$.

It is well known (see, e.g. \cite{MX2016} and the references therein) 
that the energy norm $\| \circ \|_\e$ defined in (\ref{eq:energy})
is deficient in the sense that it does 
not ``see the layers'':
as $\e \rightarrow 0$, it holds that
\[
\Vert u^S_{\e} \Vert_{\e} = O(1) \: \text{ while } \Vert u^{\pm}_{\e} \Vert_{\e} = O(\e^{1/2}).
\]
A correctly \emph{balanced} norm 
should yield $\Vert u^S_{\e} \Vert_{B} = O(1)=\Vert u^{\pm}_{\e} \Vert_{B}$.
The so-called \emph{balanced norm}
$\| \circ \|_B$ defined in the following expression 
is such a norm:
\begin{equation}\label{eq:balanced_norm}
\Vert w \Vert_{B}^2 := \e \Vert w' \Vert^2_{L^2(I)} +  \Vert w \Vert^2_{L^2(I)}.
\end{equation}
Unfortunately, 
the bilinear form associated with the weak formulation (\ref{eq:variational}) 
is \emph{not} coercive with respect to this norm, 
and standard numerical analysis techniques fail in proving exponential convergence 
with respect to this norm. In \cite{MX2016} this was by-passed through an alternative analysis 
(see \cite{MX2016} for details) and the following was shown.

\begin{proposition}[{\cite[Thm.~{2.6} and Cor.~{2.7}]{MX2016}}]
\label{prop:balanced} 
Assume that (\ref{eq:analyticfb}) holds and 
let $u_{\e}$ be the solution of (\ref{eq:variational}).

Then, 
there exists $\tilde{\kappa}_0>0$ (depending only on $b$ and $f$) 
such that for every 
$\kappa \in (0,\tilde{\kappa}_0)$ 
and 
every $p\in\N$,
the following holds.

Denoting by $u^{FEM}_{\e} \in V^p_0(\kappa)$ 
the Galerkin Finite-Element solution of (\ref{eq:discrete}),
there exist positive constants  
$C$, $\beta$ that are independent of $\e$ and $p$, 
such that
\begin{align}
\label{eq:balanced}
\left\{ \e^{1/2} \|u_{\e}' - (u^{FEM}_{\e})'\|_{L^2(I)} 
	+ \|u_{\e} - u^{FEM}_{\e}\|_{L^2(I)} \right\} 
\leq &\, C e^{-\beta p},
\\
\|u_{\e} - u^{FEM}_{\e}\|_{L^\infty(I)}
\leq 
&\, C e^{-\beta p}.
\label{eq:linfty}
\end{align}
\end{proposition}

The proof of the above proposition is again constructive, 
as it is based on the proof of Proposition \ref{prop:Melenk1997}
(see \cite{MX2016}).

Finally, 
for the approximation of the explicit boundary layer expressions from Remark \ref{rem:blexp}
we recall the following result from \cite{SS1996}.
We state our result for the boundary layer function
$\tilde{u}^-_{\e}(x) = \exp((1+x)/e)$ for $x\in(-1,1)$,
corresponding to the left boundary point,
and corresponding to $b=1$ in Remark \ref{rem:blexp}.
\begin{proposition}[{{\cite[Thm. 5.1, Cor. 5.1]{SS1996}, 
\cite[Thm. 3.74, Cor. 3.77]{Schwab1998}}}]
\label{prop:hpexp}
For $\e\in(0,1]$ and $p\in\N$, 
let the mesh $\Delta$ 
be as follows:
\begin{equation}
\label{eq:blmeshdegree}
\Delta = 
\begin{cases}
	\{ -1, -1+\kappa\tilde{p}\e, 1 \},
	& \text{ if } \kappa\tilde{p}\e <2, \\
	\{ -1, 1 \},
	& \text{ if } \kappa\tilde{p}\e \geq2,
\end{cases}
\end{equation}
for $\tilde{p} := p+\tfrac12$ and 
constants $0<\kappa_1$ and $\kappa_1\leq \kappa<4/e =: \kappa_0$ 
      which are independent of $p$ and $\e$.

Then, 
with $\tilde{u}^-_{\e}(x) = \exp((1+x)/e)$ for $x\in(-1,1)$ 
as defined in Remark \ref{rem:blexp} with $b=1$,
there exists $v\in \mathcal{S}^p(\Delta)$ with $v(\pm1) = \tilde{u}^-_{\e}(\pm1)$ 
and 
\begin{equation}
\label{eq:blrate}
\e^{1/2} \normc[L^2(I)]{ (\tilde{u}^-_{\e})' -v'}
	+ \e^{-1/2} \normc[L^2(I)]{ \tilde{u}^-_{\e} -v} 
	+ \normc[L^\infty(I)]{\tilde{u}^-_{\e}- v} 
	\leq C \exp(-\beta p)
	,
\end{equation}
for constants $C,\beta>0$ 
independent of $p$ and $\e$.

\begin{remark}
\label{rem:factorialconv}
In \cite{SS1996,Schwab1998},
for the case $\kappa\tilde{p}\e < 2$,
it is shown to be sufficient to use polynomial degree $1$ 
in the element $(-1+\kappa\tilde{p}\e,1)$.
Also, $\beta$ is specified explicitly.
For the case $\kappa\tilde{p}\e \geq 2$,
it is shown that the error converges faster than exponentially.
There exists $C>0$ such that the error bounds also hold if we replace
$C \exp(-\beta p)$ by $C \exp(- \tilde{p} \log(2\tilde{p}\e/e) )$.
\end{remark}
\end{proposition}
\section{Neural Network Definitions}
\label{sec:nn}
As usual (e.g. \cite{PV2018,OPS2020,OS2023}), 
we define a neural network (NN) in terms of its weight matrices and bias vectors.
We distinguish between a neural network and the function it realizes,
called \emph{realization} of the NN,
which is the composition of parameter-dependent affine transformations
and nonlinear activations.
We recall some NN formalism in the notation of
\cite[Section 2]{PV2018}.

\begin{definition}[{\cite[Definition 2.1]{PV2018}}]
\label{def:NeuralNetworks}
For $d,L\in\N$, a \emph{neural network $\Phi$} 
with input dimension $d \geq 1$ and number of layers $L\geq 1$, 
comprises 
a finite sequence of matrix-vector tuples, i.e.
\begin{align*}
\Phi = ((A_1,b_1),(A_2,b_2),\ldots,(A_L,b_L)).
\end{align*}

For $N_0 := d$ and \emph{numbers of neurons $N_1,\ldots,N_L\in\N$ per layer}, 
for all $\ell=1,\ldots, L$ it holds that
$A_\ell\in\R^{N_\ell \times N_{\ell-1} }$ and
$b_\ell\in\R^{N_\ell}$.

For a NN $\Phi$ and an activation function $\varrho: \R \to \R$, 
we define the associated
\emph{realization} of $\Phi$ as the function
\begin{align*}
\realiz(\Phi): \R^d\to\R^{N_L} : x \to x_L,
\end{align*}
where
\begin{align*}
x_0 & := x,
\\
x_\ell & := \varrho( A_\ell x_{\ell-1} + b_\ell ),
\qquad\text{ for }\ell=1,\ldots,L-1,
\\
x_L & := A_L x_{L-1} + b_L.
\end{align*}
Here $\varrho$ acts componentwise on vector-valued inputs, 
$\varrho(y) = (\varrho(y_1), \dots, \varrho(y_m))$ 
for all $y = (y_1, \dots, y_m) \in \R^m$.
We call the layers indexed by $\ell=1,\ldots,L-1$ \emph{hidden layers}, 
in those layers the activation function is applied.
No activation is applied in the last layer of the NN.

We refer to $\depth(\Phi) := L$ as the \emph{depth} of $\Phi$
and call $\size(\Phi) := \sum_{\ell=1}^L \norm[0]{A_\ell} + \norm[0]{b_\ell}$
the \emph{size} of $\Phi$,
which is the number of nonzero components 
in the weight matrices $A_\ell$ and the bias vectors $b_\ell$.
Furthermore, we call $d$ and $N_L$ the
\emph{input dimension} and the \emph{output dimension}.
\end{definition}

Some related works,
e.g. \cite{DLM2021}, 
use the width as a measure for the complexity of a NN,
which is defined as $\max_{\ell=0}^L N_\ell$.
Note that in each layer of a fully connected NN
the number of nonzero weights can be as large as the width squared.

We will refer to NNs with
only activation function $\varrho$ as strict $\varrho$-NNs,
or simply as $\varrho$-NNs.
This includes NNs of depth $1$, 
which do not have hidden layers and 
which exactly realize affine transformations.
\section{ReLU Neural Network Approximations}
\label{sec:relunn}
In this section, 
we consider 
the approximation of univariate functions on bounded intervals
by neural networks 
with
the ReLU activation function
$
	\rho: \R \to \R: x \mapsto \max\{0, x\}
$.
In Proposition \ref{prop:relupwpolynom} below,
we recall the ReLU NN approximation of continuous, piecewise polynomial functions
from \cite[Proposition 3.11]{OS2023}.\footnote{
The result in \cite[Proposition 3.11]{OS2023}
is stated for different polynomial degrees $p_1,\ldots,p_N \in\N$ 
in the elements $I_1,\ldots,I_N$ of the partition.
Here, we only state that result for the special case 
that $p_1 = \cdots = p_N = p \in\N$.}
It allows us to transfer the finite element approximation
results from Section \ref{sec:approx}
and obtain approximation rate bounds for ReLU NNs
in Propositions \ref{prop:relubalanced} and \ref{prop:reluexp}.
\begin{proposition}[{\cite[Proposition 3.11]{OS2023}}]
\label{prop:relupwpolynom}
For $-\infty < a < b < \infty$, let $I := (a,b)$.
For all $N\in\N$, all $p\in\N$,
all partitions 
$\Delta = \{ x_j \}_{j=0}^N$ of $I$ into $N$ open, disjoint, connected subintervals 
$I_j = (x_{j-1},x_j)$ of length $h_j = x_j - x_{j-1}$, $j=1,\ldots,N$,
$h = \max_{j=1}^N h_j$,
and for all $v\in \mathcal{S}^p(\Delta)$,
\footnote{The definition of $\mathcal{S}^p(\Delta)$ in \eqref{eq:Sp}
also applies to general intervals $I = (a,b)$ instead of $I = (-1,1)$.}
for all relative tolerances $\tau \in (0,1)$ 
there exists a ReLU NN $\Phi^{v,\Delta,p}_{\tau}$ 
such that for all $1\leq r, r' \leq \infty$ there holds
\begin{align}
\label{eq:relupwpolynom1}
(2/h_i)^{1-t} \snormc[W^{t,r}(I_i)]{v-\realiz\left(\Phi^{v,\Delta,p}_{\tau}\right)}
	\leq &\, 
		\tau \tfrac12 (2/h_i)^{1+1/r'-1/r} \min_{u\in\mathcal{P}_{p}: \atop u'' = v''|_{I_i}} \normc[L^{r'}(I_i)]{u},
	\\\nonumber
	&\, \qquad \text{ for all } i=1,\ldots,N \text{ and } t = 0,1,
\end{align}
\begin{align}
	\label{eq:relupwpolynom2}
\tfrac1h \normc[L^{r}(I)]{v-\realiz\left(\Phi^{v,\Delta,p}_{\tau}\right)}
	\leq &\,
	\snormc[W^{1,r}(I)]{v-\realiz\left(\Phi^{v,\Delta,p}_{\tau}\right)}
	\leq 
		\tfrac12 \tau \snormc[W^{1,r}(I)]{v}	
	,
\end{align}
\begin{align*}
\depth\left(\Phi^{v,\Delta,p}_{\tau}\right)
	\leq &\, C (1+\log_2(p)) \log_2(1/\tau)
		+ C (1+\log_2(p))^3 
	,
	\\
\size\left(\Phi^{v,\Delta,p}_{\tau}\right)
	\leq &\, C N p \big( 1 + \log_2(1/\tau) + \log_2(p) \big)
	,
\end{align*}
for a constant $C>0$ 
which is independent of $I$, $N$, $p$, $\Delta$, $\tau$ and $v$.

In addition, it holds that 
$\realiz\left( \Phi^{v,\Delta,p}_{\tau} \right)(x_j)=v(x_j)$ 
for all $j\in\{0,\ldots,N\}$.
The weights and biases in the hidden layers are independent of $v$.
The weights and biases in the output layer 
are linear combinations of the function values of $v$ in the Clenshaw--Curtis points 
in $I_i$ for $i=1,\ldots,N$.
\end{proposition}

As a direct corollary of 
Propositions \ref{prop:balanced} and \ref{prop:relupwpolynom}
we obtain: 
\begin{proposition}
\label{prop:relubalanced}
Assume that (\ref{eq:analyticfb}) holds. 
For $\e\in(0,1]$, let $u_{\e}$ be the solution of (\ref{eq:variational}).

Then, 
there exists $\tilde{\kappa}_0>0$ (depending only on $b$ and $f$) 
such that for every 
$\kappa \in (0,\tilde{\kappa}_0)$
and 
$p\in\N$
there exists a ReLU NN
$\Phi^{FEM,\kappa,p}_{\e}$
such that, 
with positive constants $C$, $\beta$, 
independent of $\e$ and $p$,
it holds that
\begin{align}
\label{eq:relubalanced}
\left\{ \e^{1/2} \|u_{\e}' - \realiz(\Phi^{FEM,\kappa,p}_{\e})'\|_{L^2(I)} 
	+ \|u_{\e} - \realiz(\Phi^{FEM,\kappa,p}_{\e}) \|_{L^2(I)} \right\} 
\leq &\, C e^{-\beta p},
\\
\|u_{\e} - \realiz(\Phi^{FEM,\kappa,p}_{\e}) \|_{L^\infty(I)}
\leq &\, C e^{-\beta p},
\label{eq:relulinfty}
\end{align}
and $\realiz(\Phi^{FEM,\kappa,p}_{\e})(\pm1) = 0$.

For a constant $\tilde{C} = \tilde{C}(\beta)>0$ depending only on $\beta$,
the network depth and size are bounded as follows: 
\begin{align}
\label{eq:relubldepthsize}
\depth\left(\Phi^{FEM,\kappa,p}_{\e}\right)
	\leq &\, \tilde{C} p (1+\log_2(p))
	,
	\qquad
\size\left(\Phi^{FEM,\kappa,p}_{\e}\right)
	\leq
		\tilde{C} p^2
. 
\end{align}
The weights and biases in the hidden layers are independent of $u_{\e}$
and depend only on $\kappa$, $p$, $\e$ and $\beta$.
\end{proposition}

\begin{proof}
We apply Proposition \ref{prop:relupwpolynom} 
to
$u^{FEM}_{\e}\in V^p_0(\kappa)$ 
from 
Proposition \ref{prop:balanced},
with accuracy parameter $\tau = e^{-\beta p}$
for $\beta$ given in Proposition \ref{prop:balanced}
and the Spectral Boundary Layer mesh 
$\Delta := \Delta_{BL}(\kappa,p)$ 
from Definition \ref{def:SBL},
i.e. 
if $\kappa p \e<1/2$, then the number of elements is $N = 3$,
whereas if $\kappa p \e\geq 1/2$, then $N=1$.
We define 
$\Phi^{FEM,\kappa,p}_{\e} := \Phi^{u^{FEM}_{\e},\Delta,p}_{\tau}$.
We obtain from \eqref{eq:relupwpolynom1},
with $t=0$, $r = r' \in \{2,\infty\}$ and $u = v$,
that on all elements $I_j \in \Delta_{BL}(\kappa,p)$, $j=1,\ldots,N$, 
it holds that
$\normc[L^r(I_j)]{ u^{FEM}_{\e} - \realiz(\Phi^{FEM,\kappa,p}_{\e}) }
	\leq \tfrac12 \tau \normc[L^r(I_j)]{ u^{FEM}_{\e} }$,
and thus 
\begin{align*}
\normc[L^2(I)]{ u^{FEM}_{\e} - \realiz(\Phi^{FEM,\kappa,p}_{\e}) }
	\leq &\, \tfrac12 \tau \normc[L^2(I)]{ u^{FEM}_{\e} }
	\\
	\leq &\, 
		\tfrac12 \tau \left( \normc[L^2(I)]{ u_{\e} } + \normc[L^2(I)]{ u_{\e} - u^{FEM}_{\e} } \right)
	\\
	\leq &\, 
		\tfrac12 e^{-\beta p} \left( C + C e^{-\beta p} \right)
	\leq
		C e^{-\beta p}
,
\end{align*}
where we used Lemma \ref{lem:normsue} and \eqref{eq:balanced} in the third step.
Here, and in the remainder of the proof,
$\beta$ is as in Proposition \ref{prop:balanced}
and $C$ denotes 
a generic positive constant
which is independent of $\e$ and $p$,
but may be different at each appearance.
We have the same result as above, also in the maximum norm:
\begin{align*}
\normc[L^\infty(I)]{ u^{FEM}_{\e} - \realiz(\Phi^{FEM,\kappa,p}_{\e}) }
	\leq &\, 
		\tfrac12 e^{-\beta p} \left( \normc[L^\infty(I)]{ u_{\e} } + C e^{-\beta p} \right)
	\leq
		C e^{-\beta p}
.
\end{align*}
From \eqref{eq:relupwpolynom2}
we obtain
\begin{align*}
\normc[L^2(I)]{ (u^{FEM}_{\e})' - \realiz(\Phi^{FEM,\kappa,p}_{\e})' }
	\leq \tfrac12 \tau \normc[L^2(I)]{ (u^{FEM}_{\e})' }.
\end{align*}
Combined with Lemma \ref{lem:normsue} and \eqref{eq:balanced}, this gives
\begin{align*}
\normc[L^2(I)]{ (u^{FEM}_{\e})' - \realiz(\Phi^{FEM,\kappa,p}_{\e})' }
	\leq &\, \tfrac12 \tau \normc[L^2(I)]{ (u^{FEM}_{\e})' }
	\\
	\leq &\, 
		\tfrac12 \tau \left( \normc[L^2(I)]{ u_{\e}' } + \normc[L^2(I)]{ u_{\e}' - (u^{FEM}_{\e})' } \right)
	\\
	\leq &\, 
		\tfrac12 e^{-\beta p} \left( C \e^{-1/2} + C \e^{-1/2} e^{-\beta p} \right)
	\\
	\leq &\, 
		C \e^{-1/2} e^{-\beta p}
.
\end{align*}
Using the triangle inequality to combine these estimates with 
Equations \eqref{eq:balanced}--\eqref{eq:linfty}
finishes the proof of Equations \eqref{eq:relubalanced}--\eqref{eq:relulinfty}. 
By Proposition \ref{prop:relupwpolynom}, it also holds that 
$\realiz(\Phi^{FEM,\kappa,p}_{\e})(\pm1) = u^{FEM}_{\e}(\pm1) = 0 = u_{\e}(\pm1)$.

As upper bounds on the network depth and size,
we obtain from Proposition \ref{prop:relupwpolynom}
\begin{align*}
\depth\left(\Phi^{FEM,\kappa,p}_{\e}\right)
	\leq &\, C (1+\log_2(p)) \log_2(1/\tau)
		+ C (1+\log_2(p))^3 
	\\
	\leq &\, 
		\tilde{C} p (1+\log_2(p))
	,
	\\
\size\left(\Phi^{FEM,\kappa,p}_{\e}\right)
	\leq &\, C N p \big( 1 + \log_2(1/\tau) + \log_2(p) \big)
	\\
	\leq &\, \tilde{C} p^2
,
\end{align*}
for $\tilde{C} = \tilde{C}(\beta)>0$ depending only on $\beta$.
In the last step, we used that $N\leq 3$.
\end{proof}

By the same arguments as in the proof of
Proposition \ref{prop:relubalanced},
we obtain from Propositions \ref{prop:hpexp} and \ref{prop:relupwpolynom}
the following result on the approximation of 
exponential boundary layer functions.
\begin{proposition}
\label{prop:reluexp}
There exists $\tilde{\kappa}_1>0$
such that for every 
$\kappa \in (\tilde{\kappa}_1,4/e)$
and 
$p\in\N$
there exists a ReLU NN
$\Phi^{\exp,\kappa,p}_{\e}$
such that, 
with positive constants $C$, $\beta$, 
independent of $\e$ and $p$,
it holds that
\begin{align}
\label{eq:reluexph1}
\e^{1/2} \| -\exp(-\cdot/\e)/\e - \realiz(\Phi^{\exp,\kappa,p}_{\e})'\|_{L^2((0,1))} 
\leq &\, C e^{-\beta p},
\\
\label{eq:reluexpl2}
\| \exp(-\cdot/\e) - \realiz(\Phi^{\exp,\kappa,p}_{\e}) \|_{L^2((0,1))}
\leq &\, C e^{-\beta p},
\\
\| \exp(-\cdot/\e) - \realiz(\Phi^{\exp,\kappa,p}_{\e}) \|_{L^\infty((0,1))}
\leq &\, C e^{-\beta p}.
\label{eq:reluexplinfty}
\end{align}

For a constant $\tilde{C} = \tilde{C}(\beta)>0$ depending only on $\beta$,
the NN depth and size are bounded as follows: 
\begin{align}
\label{eq:reluexpdepthsize}
\depth\left(\Phi^{\exp,\kappa,p}_{\e}\right)
	\leq &\, \tilde{C} p (1+\log_2(p))
	,
	\qquad
\size\left(\Phi^{\exp,\kappa,p}_{\e}\right)
	\leq
		\tilde{C} p^2
. 
\end{align}
\end{proposition}
\begin{proof}
Let $\tilde{u}^-_{2\e}$ and $v$ be as in Proposition \ref{prop:hpexp},
with $2\e$ in place of $\e$.
Composing both 
$\tilde{u}^-_{2\e}$
and
$v$
with the affine transformation $P:[0,1]\to[-1,1]:x\mapsto 2x-1$
gives
$\exp(-x/\e) = \tilde{u}^-_{2\e} \circ P (x)$ for all $x\in(0,1)$ and,
for $r=2,\infty$,
\begin{align*}
\norm[L^r((0,1))]{ \exp(-x/\e) - v \circ P }
	\leq &\, \norm[L^r((-1,1))]{ \tilde{u}^-_{2\e} - v } 
	\leq C \exp( -\beta p)
	,
	\\
\e^{1/2} \norm[L^2((0,1))]{ -\tfrac{1}{\e} \exp(-x/\e) - (v \circ P)' }
	\leq &\, \e^{1/2} 
         \norm[L^2((-1,1))]{ (\tilde{u}^-_{2\e})' - v' } \norm[L^\infty((0,1))]{ P' }
	\\
	\leq &\, C \exp( -\beta p) \cdot 2
	= C \exp( -\beta p)
	.
\end{align*}
Now, we can apply Proposition \ref{prop:relupwpolynom}
to $v \circ P$ and $\tilde\Delta = \{ 0, \kappa\tilde{p}\e/2, 1 \}$
with accuracy $\tau = \exp(-\beta p)$
to obtain the existence of a ReLU NN
$\Phi^{\exp,\kappa,p}_{\e} := \Phi^{v\circ P, \tilde\Delta,p}_{\tau}$ 
which satisfies for $r=2,\infty$
\begin{align*}
\norm[L^r((0,1))]{ v \circ P - \realiz(\Phi^{\exp,\kappa,p}_{\e}) }
	\leq &\, \tfrac12 \exp(-\beta p) \norm[L^r((0,1))]{ v \circ P },
	\\
\norm[L^2((0,1))]{ (v \circ P)' - \realiz(\Phi^{\exp,\kappa,p}_{\e})' }
	\leq &\, \tfrac12 \exp(-\beta p) \norm[L^2((0,1))]{ (v \circ P)' }
	,
\end{align*}
from which we obtain the desired error bounds 
using the same arguments as in the proof of Proposition \ref{prop:relubalanced},
using \eqref{eq:blnorms}.

We also find the bounds on the network depth and size 
as in Proposition \ref{prop:relubalanced}:
\begin{align*}
\depth\left(\Phi^{\exp,\kappa,p}_{\e}\right)
	\leq &\, C (1+\log_2(p)) \log_2(1/\tau)
		+ C (1+\log_2(p))^3 
	\leq
		\tilde{C} p (1+\log_2(p))
	,
	\\
\size\left(\Phi^{\exp,\kappa,p}_{\e}\right)
	\leq &\, C N p \big( 1 + \log_2(1/\tau) + \log_2(p) \big)
	\leq
		\tilde{C} p^2
,
\end{align*}
for $\tilde{C} = \tilde{C}(\beta)>0$ depending only on $\beta$.
\end{proof}

\section{Spiking Neural Network Approximation}
\label{sec:spiking}
So far, we obtained expression rate bounds of strict ReLU NNs
where all activations are ReLUs, 
for solutions of the singularly perturbed, 
two-point boundary value problem \eqref{eq:de}--\eqref{eq:bc}.
Approximation rates for ReLU NNs transfer to so-called 
\emph{spiking neural networks} (SNNs) which
are at the core of some models of so-called \emph{neuromorphic computing}
(see, e.g., \cite{SWBCPG2022} and the references there).
Results in this direction go back several decades,
see e.g. \cite{Maass1997a,Maass1997b}
and the references there.
More recently,
Algorithms 1 and 2 in \cite{SWBCPG2022} 
produce for every strict ReLU NN $\Phi$
a SNN $\spike(\Phi)$ whose realization 
$\realiz(\spike(\Phi))$
is identical to the input-output map 
$x\mapsto \realiz(\Phi)(x)$ of the ReLU NN $\Phi$
up to an affine transformation of the input, 
see Proposition \ref{prop:correspspikingrelu} for details.

We use \cite[Alg.~1 and 2]{SWBCPG2022} to deduce from the 
approximation rate bounds in Section~\ref{sec:relunn} 
corresponding results for SNNs 
in terms of the number of nonzero weights in the SNN. 
We proceed as follows.
After defining SNNs, we recall the exact mapping 
from ReLU NNs to SNNs 
in 
Algorithms \ref{algo:rescalerelu} and \ref{algo:relunntosnn} below.
In Proposition \ref{prop:correspspikingrelu}
we estimate the size of the resulting SNN in terms of the size of the ReLU NN.
SNN approximation of solutions to the singularly perturbed model problem 
in Section~\ref{sec:model} is the topic of Theorem~\ref{thm:spikingbalanced}.
\subsection{Spiking Neural Network Definitions}
\label{sec:defspiking}
As in \cite{SWBCPG2022}, we consider 
SNNs with \emph{integrate-and-fire neurons} in the hidden layers,
in which each hidden layer neuron 
fires exactly once during each evaluation of the network.
The output of a hidden layer neuron $i$ in layer $\ell$ 
is the spiking time 
$(t_\ell)_i \in [t^{\min}_\ell,t^{\max}_\ell]$.
The spiking time is defined (and computed) 
as the first time $t \geq t^{\min}_\ell$ 
at which the \emph{voltage trajectory} 
$(V_\ell)_i(t)$ of neuron $i$ in layer $\ell$
attains the threshold value $(\vartheta_\ell)_i$, 
as detailed in the following definition.
The output layer consists of \emph{integration neurons}, which do not fire.
The output of each such neuron $i$ in the last layer
is the voltage at the final time $(V_L)_i(t^{\max}_L)$.

\begin{definition}[Spiking neural network (SNN)] (\cite[Section 2.1]{SWBCPG2022})
\label{def:spikingnn}
For $d,L\in\N$, a \emph{spiking neural network $\Phi$} 
with input dimension $d \geq 1$ and number of layers $L\geq 1$, 
is given by a finite sequence
of matrix-vector-vector-number-number tuples, i.e.
\begin{align*}
\Phi = &\, \big((J_1,\vartheta_1,\alpha_1,t^{\min}_1,t^{\max}_1),
\ldots,
(J_{L-1},\vartheta_{L-1},\alpha_{L-1},t^{\min}_{L-1},t^{\max}_{L-1}),
\\
&\,
(J_L,\alpha_L,t^{\min}_L,t^{\max}_L)\big)
,
\end{align*}
where in the last tuple, the vector $\vartheta_L$ is omitted.
For $N_0 := d$ 
and \emph{numbers of neurons $N_1,\ldots,N_L\in\N$ per layer}, 
for all $\ell=1,\ldots, L$ it holds that
$J_\ell\in\R^{N_\ell \times N_{\ell-1} }$,
$\vartheta_\ell, \alpha_\ell \in\R^{N_\ell}$
and
$t^{\min}_\ell, t^{\max}_\ell \in\R$,
with the exception that we do not consider $\vartheta_L$.
In addition, we require that
$t^{\max}_{\ell-1} = t^{\min}_{\ell}$ for all $\ell=1,\ldots,L$
and that
\[
0 = t^{\min}_0 < t^{\min}_1 = 1 < \cdots < t^{\min}_L
\;\;\mbox{and}\;\;
t^{\max}_0 = 1 < t^{\max}_1 < \cdots < t^{\max}_{L-1} = t^{\max}_L.
\]

The input of $\Phi$ 
comprises the firing times $t_0 \in [t^{\min}_0, t^{\max}_0]^d$ 
of the neurons in the input layer.
For all $\ell=1,\ldots,{L-1}$ and $i=1,\ldots,N_\ell$,
the spiking time $(t_\ell)_i \in [t^{\min}_\ell, t^{\max}_\ell]$
of neuron $i$ in layer $\ell$
is defined as 
the first time $t \geq t^{\min}_\ell$ 
at which the 
\emph{voltage trajectory} $(V_\ell)_i(t)$
attains or exceeds the \emph{threshold} $(\vartheta_\ell)_i \in\R$,
where $(V_\ell)_i(t)$ is defined by 
$(V_\ell)_i(t^{\min}_{\ell-1}) = 0$
and the following ODE, which holds for all 
$t\in(t^{\min}_{\ell-1},t^{\max}_{\ell})$:
\begin{align}
\label{eq:voltageode}
\tfrac{\dd}{\dd t} (V_\ell)_i(t)
	= &\, (\alpha_\ell)_i H( t - t^{\min}_{\ell-1} )
		+ \sum_{j=1}^{N_{\ell-1}} (J_\ell)_{ij} H( t - (t_{\ell-1})_j )
		+ (I_\ell)_i(t)
.
\end{align}
Here, $H:\R\to\R$ denotes the Heaviside function,
defined by $H(x) = 1$ for $x>0$ and $H(x) = 0$ else.
The values $(J_\ell)_{ij}$ are called \emph{weights}
and $(\alpha_\ell)_i$ is the \emph{slope parameter}.
In layers $\ell = 1,\ldots,L-1$, a nonnegative short \emph{pulse} $(I_\ell)_i(t)$ 
is used to force the neuron to spike at the latest at $t^{\max}_{\ell}$.
In the output layer $\ell = L$, 
the voltage trajectory is also defined by \eqref{eq:voltageode},
with $I_L \equiv 0$.
The output of $\Phi$ 
comprises the voltages of the neurons in the output layer
at time $t^{\max}_L = t^{\max}_{L-1}$,
which we denote by
$\realiz(\Phi)(t_0) := (V_L)(t^{\max}_L) \in\R^{N_L}$.

We refer to $\depth(\Phi) := L$ as the \emph{depth} of $\Phi$
and call 
$\size(\Phi) := \sum_{\ell=1}^L \norm[0]{J_\ell}$
the \emph{size} of $\Phi$,
which is the number of nonzero components 
in the weight matrices $J_\ell$.
Furthermore, we call $d$ and $N_L$ the
\emph{input dimension} and the \emph{output dimension}.

\end{definition}

In \cite{SWBCPG2022}, 
for neuron $i=1,\ldots,N_\ell$ in hidden layer $\ell=1,\ldots,L-1$,
the pulse is defined 
in terms of the Dirac delta distribution
as $(I_\ell)_i(t) = R \delta( t - t^{\max}_\ell)$
for some sufficiently large $R>0$.
Denoting by $(\tilde{V}_\ell)_i$ the voltage trajectory in case $(I_\ell)_i \equiv 0$,
it is sufficient to set $R = (\vartheta_\ell)_i - (\tilde{V}_\ell)_i(t^{\max}_\ell)$.

\begin{remark}\label{rmk:EqivPlse}
For any $0 < \eta \leq t^{\max}_\ell - t^{\min}_\ell$
we could equivalently consider the current pulse 
\[
(I_\ell)_i(t) = 
\begin{cases}
	0 & \text{ for } t \in (t^{\min}_{\ell-1}, t^{\max}_\ell - \eta),
	\\
	\big( (\vartheta_\ell)_i - (\tilde{V}_\ell)_i(t^{\max}_\ell) \big) / \eta
	& \text{ for } t \in ( t^{\max}_\ell - \eta, t^{\max}_\ell ),
\end{cases}
\]
such that $(V_\ell)_i(t^{\max}_\ell) = (\vartheta_\ell)_i$.
To ensure that the SNN does not fire earlier than at time $t^{\max}_\ell$,
it suffices to choose $\eta$ small enough,
e.g. such that
$\big( (\vartheta_\ell)_i - (\tilde{V}_\ell)_i(t^{\max}_\ell) \big) /\eta 
	> \max_{t\in[t^{\max}_\ell-\eta,t^{\max}_\ell]} \snorm{(\tilde{V}_\ell)_i'(t)}$.
From this we obtain that $(V_\ell)_i'(t) > 0$ 
for all $t\in ( t^{\max}_\ell - \eta, t^{\max}_\ell )$,
and thus that $(V_\ell)_i(t) < (V_\ell)_i(t^{\max}_\ell)$ for all such $t$.
For all $t < t^{\max}_\ell - \eta$, the pulse does not affect $(\tilde{V}_\ell)_i(t)$,
hence also for such $t$ the SNN does not spike.
\end{remark}

\begin{remark}\label{rmk:SpkTime}
Imposing $(t_\ell)_i \geq t^{\min}_{\ell}$
is important.
Although the voltage trajectory $(V_\ell)_i(t)$ 
may attain or exceed the threshold value $(\vartheta_\ell)_i$ 
at an earlier time,
we do not want the neuron to fire earlier than $t^{\min}_{\ell}$.
In \cite{SWBCPG2022}, 
this is interpreted as using a time-dependent threshold,
which equals the previously specified value 
for $t\geq t^{\min}_\ell$,
and a very large value 
for $t<t^{\min}_\ell$.
\end{remark}

\subsection{ReLU to Spiking Neural Network Conversion}
\label{sec:converttospiking}

Next, we state a version of 
\cite[Algorithms 1 and 2]{SWBCPG2022}
for transforming feedforward ReLU networks.
For each ReLU NN, 
the SNN produced by these algorithms
has the same input dimension, output dimension,
depth and the same layer dimensions,
see Proposition \ref{prop:correspspikingrelu}.
A large output value of a neuron from the ReLU NN
corresponds to early spiking of the corresponding SNN neuron.

We have slightly modified line 15 from
\cite[Algorithm 1]{SWBCPG2022}
to define an exact mapping from ReLU NNs to spiking NNs
without making use of a training data set.
See Remark \ref{rem:range} below.

\begin{algorithm}
\caption{
\cite[Algorithm 1]{SWBCPG2022}
The inputs are 
$d,L \in \N$, 
$N_1,\ldots,N_L\in\N$,
constants $\delta\in(0,1)$ and $B>0$
and a ReLU NN $\Phi$
which takes inputs from $[x^{\min},x^{\max}]^d$.
The output is a ReLU NN 
$\overline\Phi = ((\overline{A}_1,\overline{b}_1),\ldots,(\overline{A}_L,\overline{b}_L))$
with the same input dimension $d$, 
depth $L$ and 
layer dimensions $N_1,\ldots,N_L$,
which takes inputs from $[0,1]^d$
such that $\realiz(\Phi)(x) = \realiz(\overline\Phi)(\overline{x})$
for all $x\in[x^{\min},x^{\max}]^d$ and 
$\overline{x} = \tfrac{ 1 }{ x^{\max} - x^{\min} } ( x - x^{\min} (1,\ldots,1)^\top )$,
and such that for all $\ell=1,\ldots,L-1$ and $i=1,\ldots,N_\ell$ holds
$\sum_{j=1}^{N_{\ell-1}} \overline{A}_{ij} \in [-B,\delta]$.
To rescale the weights and biases of the given ReLU NN $\Phi$
without changing the NN output,
the algorithm exploits the positive homogeneity of the ReLU activation function
$\rho(\lambda x) = \lambda \rho(x)$ for all $\lambda>0$ and $x\in\R$.
}
\label{algo:rescalerelu}
\begin{algorithmic}[1] 
\Input $d,L \in \N$, 
$N_1,\ldots,N_L\in\N$,
$x^{\min},x^{\max}\in\R$, $x^{\min} < x^{\max}$,
$\delta\in(0,1)$, $B>0$
and 
a ReLU NN $\Phi = ((A_1,b_1),\ldots,$ $(A_L,b_L))$
\Output A ReLU NN 
$\overline\Phi 
	= ((\overline{A}_1,\overline{b}_1),\ldots,$ $(\overline{A}_L,\overline{b}_L))$
\For{$i=1,\ldots,N_1$}
	\label{line:fornodes-i}
	\State
	For $j=1,\ldots,N_0$:
	$(\overline{A}_1)_{ij} \leftarrow ( x^{\max} - x^{\min} ) (A_1)_{ij}$
	\label{line:scaleinputweights}
	\State
	$(\overline{b}_1)_i \leftarrow (b_1)_i + x^{\min}\sum_{j=1}^{d} (A_1)_{ij}$
	\label{line:scaleinputbias}
\EndFor
\State
For $\ell=2,\ldots,L$:
$\overline{A}_\ell \leftarrow A_\ell$, $\overline{b}_\ell \leftarrow b_\ell$, 
\label{line:initialize}
\For{$\ell=1,\ldots,L-1$} 
\label{line:forlayers-i}
	\For{$i=1,\ldots,N_\ell$}
	\label{line:fornodes-ii}
		\State
		$(c_\ell)_i \leftarrow \sum_{j=1}^{N_{\ell-1}} \overline{A}_{ij}$
		\label{line:c}
		\If{$(c_\ell)_i > 1-\delta$}
		\label{line:ifctoopos}
			\State 
			For $j=1,\ldots,N_{\ell-1}$: 
			$(\overline{A}_{\ell})_{ij} \leftarrow 
				\tfrac{1-\delta}{(c_\ell)_i} (\overline{A}_{\ell})_{ij}$
			\label{line:poscprevweight}
			\State 
			$(\overline{b}_{\ell})_{i} \leftarrow 
				\tfrac{1-\delta}{(c_\ell)_i} (\overline{b}_{\ell})_{i}$
			\label{line:poscbias}
			\State 
			For $k=1,\ldots,N_{\ell+1}$: 
			$(\overline{A}_{\ell+1})_{ki} \leftarrow 
				\tfrac{(c_\ell)_i}{1-\delta} (\overline{A}_{\ell+1})_{ki}$
			\label{line:poscnextweight}
		\ElsIf{$(c_\ell)_i < - B$}
		\label{line:ifctooneg}
			\State 
			For $j=1,\ldots,N_{\ell-1}$: 
			$(\overline{A}_{\ell})_{ij} \leftarrow 
				\tfrac{B}{\snorm{(c_\ell)_i}} (\overline{A}_{\ell})_{ij}$
			\label{line:negcprevweight}
			\State 
			$(\overline{b}_{\ell})_{i} \leftarrow 
				\tfrac{B}{\snorm{(c_\ell)_i}} (\overline{b}_{\ell})_{i}$
			\label{line:negcbias}
			\State 
			For $k=1,\ldots,N_{\ell+1}$: 
			$(\overline{A}_{\ell+1})_{ki} \leftarrow 
				\tfrac{\snorm{(c_\ell)_i}}{B} (\overline{A}_{\ell+1})_{ki}$
			\label{line:negcnextweight}
		\EndIf
	\EndFor
\EndFor
\State
\Return $\overline\Phi 
	\leftarrow ((\overline{A}_1,\overline{b}_1),\ldots,$ $(\overline{A}_L,\overline{b}_L))$
\label{line:rescaledreturn}
\end{algorithmic}
\end{algorithm}

\begin{algorithm}
\caption{
\cite[Algorithm 2]{SWBCPG2022}
The inputs are 
$d,L \in \N$, 
$N_1,\ldots,N_L\in\N$,
constants $\delta\in(0,1)$, $B>0$
and a ReLU NN $\Phi$
which takes inputs from $[x^{\min},x^{\max}]^d$.
The output is a spiking neural network $\spike(\Phi)$
with the same 
input dimension $d$, 
depth $L$ and 
layer dimensions $N_1,\ldots,N_L$.
First, 
the neural network weights are rescaled using Algorithm \ref{algo:rescalerelu}.
Then, a spiking neural network is defined 
such that 
for all $x\in[x_{\min},x_{\max}]^d$,
with 
$\overline{x} = \tfrac{ 1 }{ x^{\max} - x^{\min} } ( x - x^{\min} (1,\ldots,1)^\top )$
and $t_0 = (1,\ldots,1)^\top - \overline{x}$,
for all $\ell=1,\ldots,L-1$,
the output 
$(\overline{x}_\ell)_i
	:= \realiz(((\overline{A}_1,\overline{b}_1),\ldots,
	(\overline{A}_\ell,\overline{b}_\ell),
	(I_{N_\ell\times N_\ell},0_{N_\ell})))(\overline{x})$ 
of neuron $i$ in layer $\ell$ of the rescaled ReLU NN
after applying ReLU activation
corresponds to a spiking time 
$(t_\ell)_i = t^{\max}_\ell - (\overline{x}_\ell)_i$,
and such that
$\realiz(\Phi)(x) = \realiz(\spike(\Phi))( t_0 )$.}
\label{algo:relunntosnn}
\begin{algorithmic}[1] 
\Input 
$d,L \in \N$, $N_1,\ldots,N_L\in\N$,
$x^{\min},x^{\max}\in\R$, $x^{\min} < x^{\max}$,
$\delta\in(0,1)$, $B>0$
and 
a ReLU NN $\Phi = ((A_1,b_1),\ldots,$ $(A_L,b_L))$
\Output 
An SNN 
$\spike(\Phi) = ((J_1,\vartheta_1,\alpha_1,t^{\min}_1,t^{\max}_1),\ldots,
(J_L,\alpha_L,t^{\min}_L,t^{\max}_L))$
\State
Compute
$\overline\Phi 
	\leftarrow ((\overline{A}_1,\overline{b}_1),\ldots,$ $(\overline{A}_L,\overline{b}_L))$
with Algorithm \ref{algo:rescalerelu}
\label{line:rescalerelu}
\State 
$t^{\min}_0 \leftarrow 0$,
$t^{\max}_0 \leftarrow 1$
\label{line:setfirstparams}
\For{$\ell=1,\ldots,L-1$} 
\label{line:forlayers-ii}
	\State
	$X_\ell \leftarrow \max_{\overline{x}\in[0,1]^d} 
	\| \realiz(((\overline{A}_1,\overline{b}_1),\ldots,
	(\overline{A}_\ell,\overline{b}_\ell),
	(I_{N_\ell\times N_\ell},0_{N_\ell})))(\overline{x}) \|_\infty$,
	where $I_{N_\ell\times N_\ell} \in \R^{N_\ell \times N_\ell}$ denotes the identity matrix,
	and $0_{N_\ell} \in \R^{N_\ell}$ the zero vector.
	\label{line:maxoutput}
	\State 
	$t^{\min}_\ell \leftarrow t^{\max}_{\ell-1}$,
	$t^{\max}_\ell \leftarrow t^{\max}_{\ell-1} + X_\ell$,
	$\alpha_\ell \leftarrow (1,\ldots,1)^\top \in \R^{N_\ell}$
	\label{line:tminmax}
	\For{$i=1,\ldots,N_\ell$}
	\label{line:fornodes-iii}
		\State
		For $j=1,\ldots,N_{\ell-1}$:
		$(J_\ell)_{ij} \leftarrow (\alpha_\ell)_i (\overline{A}_\ell)_{ij} 
			/ \Big( 1 - \sum_{j=1}^{N_{\ell-1}} (\overline{A}_\ell)_{ij} \Big)$
		\label{line:jdef}
		\State
		$(\vartheta_\ell)_i \leftarrow (\alpha_\ell)_i ( t^{\max}_\ell - t^{\min}_{\ell-1} ) 
			+ \sum_{j=1}^{N_{\ell-1}} (J_\ell)_{ij} ( t^{\max}_\ell - t^{\min}_{\ell} ) $
		\newline
			$- \Big( (\alpha_\ell)_i + \sum_{j=1}^{N_{\ell-1}} (J_\ell)_{ij} \Big) (\overline{b}_\ell)_i$
		\label{line:varthetadef}
	\EndFor
\EndFor
\State
$t^{\min}_L \leftarrow t^{\max}_{L-1}$,
$t^{\max}_L \leftarrow t^{\max}_{L-1}$,
$J_L \leftarrow \overline{A}_L$,
$\alpha_L \leftarrow \overline{b}_L / ( t^{\max}_{L-1} - t^{\min}_{L-1} )$
\label{line:lastlayer}
\State
\Return $\spike(\Phi) \leftarrow ((J_1,\vartheta_1,\alpha_1,t^{\min}_1,t^{\max}_1),\ldots,
(J_L,\alpha_L,t^{\min}_L,t^{\max}_L))$
\label{line:return}
\end{algorithmic}
\end{algorithm}

\begin{remark}
\label{rem:range}
In Line \ref{line:maxoutput} of Algorithm \ref{algo:relunntosnn},
we slightly deviate from Line 15 in \cite[Algorithm 1]{SWBCPG2022}.
Because the ReLU NN 
$((\overline{A}_1,\overline{b}_1),\ldots,
	(\overline{A}_\ell,\overline{b}_\ell),
	(I_{N_\ell\times N_\ell},$ $0_{N_\ell}))$
realizes a continuous function
and we only consider inputs $\overline{x}$ from the compact set $[0,1]^d$,
the maximum in Line \ref{line:maxoutput} exists and is finite.
We will use this theoretical value of $X_\ell$.
We note that in \cite{SWBCPG2022},
it is argued that computing the maximum over 
(a statistically representative subset of) the training data 
is sufficient in practice.
See part (iv) of \cite[Section 4.1]{SWBCPG2022}.
By defining $X_\ell$ to be the theoretical maximum, 
rather than an empirical maximum,
it is not necessary anymore to multiply it with a factor $(1+\zeta)$ for $\zeta>0$
to obtain an upper bound that also holds 
for (practically) all inputs $\overline{x}\in[0,1]^d$.
This multiplicative factor was used in part (iii) of \cite[Section 4.2]{SWBCPG2022},
we do not use it here.

Another difference and simplification with respect to \cite{SWBCPG2022}
is that we are only interested in transforming feedforward neural networks
without convolutional layers, batch normalization and max pooling.
See \cite{SWBCPG2022} for the transformation of such features.
\end{remark}

\begin{proposition}[{\cite[Theorem and Corollary in Section 2.1]{SWBCPG2022}}]
\label{prop:correspspikingrelu}
Let
$d,L \in \N$, 
$N_1,\ldots,N_L\in\N$,
$x^{\min},x^{\max}\in\R$, $x^{\min}<x^{\max}$,
$\delta\in(0,1)$, $B>0$
and let $\Phi = ((A_1,b_1),\ldots,$ $(A_L,b_L))$
be a ReLU NN.

Then, the SNN $\spike(\Phi)$ which is the output of Algorithm \ref{algo:relunntosnn}
has input dimension $d$, depth $L$ and layer dimensions $N_1,\ldots,N_L$ 
and satisfies, for all inputs $x\in[x^{\min},x^{\max}]^d$,
with $\overline{x} = \tfrac{ 1 }{ x^{\max} - x^{\min} } ( x - x^{\min} (1,\ldots,1)^\top )$
and $t_0 = (1,\ldots,1)^\top - \overline{x}$,
that 
$\realiz(\Phi)(x) = \realiz(\spike(\Phi))( t_0 )$.
In addition, 
$\size(\spike(\Phi)) \leq \size(\Phi)$.
\end{proposition}
\begin{proof}
The formula for the realization was proved in \cite{SWBCPG2022}.

The fact that the ReLU NN $\overline{\Phi}$,
which is the output of Algorithm \ref{algo:rescalerelu} applied to a ReLU NN $\Phi$,
has the same input dimension and layer dimensions as $\Phi$
can be observed from the lines in the algorithm in which the weight matrices are initialized.
These are Lines \ref{line:scaleinputweights} and \ref{line:initialize}.
Other lines of the algorithm do not change the sizes of the weight matrices.
From Line \ref{line:rescaledreturn} 
we see that the network $\overline\Phi$ returned by the algorithm 
has the same number of layers $L$ as the input network $\Phi$.
The same ideas apply to Algorithm \ref{algo:relunntosnn},
where we see from Lines \ref{line:jdef} and \ref{line:lastlayer},
where the weight matrices are computed,
that the input dimension and the layer dimensions of $\spike(\Phi)$
equal those of the output $\overline{\Phi}$ of Algorithm \ref{algo:rescalerelu},
and thus those of $\Phi$.
From Line \ref{line:return}, 
we observe that the number of layers of $\spike(\Phi)$ is $L$,
which is the same as that of $\overline\Phi$ and that of $\Phi$.

To prove the bound on the network size,
we first observe that in all lines which affect the weights,
which are Lines 
\ref{line:scaleinputweights}, \ref{line:initialize}, 
\ref{line:poscprevweight}, \ref{line:poscnextweight}, 
\ref{line:negcprevweight} and \ref{line:negcnextweight} 
of Algorithm \ref{algo:rescalerelu}
and Lines
\ref{line:jdef}
and
\ref{line:lastlayer}
of Algorithm \ref{algo:relunntosnn},
the sign of the weights is not changed.
In particular, the number of nonzero weights of the SNN
equals that of the ReLU NN,
which implies the desired neural network size bound.
\end{proof}

\subsection{Spiking Neural Network Solution Approximation}
\label{sec:blspiking}
As a direct consequence of 
Propositions \ref{prop:relubalanced} and \ref{prop:correspspikingrelu}
we obtain the expression rate bounds for solutions of \eqref{eq:de}--\eqref{eq:bc}
with spiking NNs.
\begin{theorem}
\label{thm:spikingbalanced}
Assume that (\ref{eq:analyticfb}) holds. 
For $\e\in(0,1]$ let $u_{\e}$ be the solution of (\ref{eq:variational}).
Recall from Proposition \ref{prop:relubalanced}
the constant
$\tilde{\kappa}_0>0$ (depending only on $b$ and $f$) 
and for all
$\kappa \in (0,\tilde{\kappa}_0)$
and 
$p\in\N$
the ReLU NN
$\Phi^{FEM,\kappa,p}_{\e}$.

Then, 
with the positive constants $C$ and $\beta$
from Proposition \ref{prop:relubalanced},
independent of $\e$ and $p$,
the SNN $\spike(\Phi^{FEM,\kappa,p}_{\e})$
constructed by Algorithm \ref{algo:relunntosnn}
satisfies
\begin{align}
\label{eq:spikingbalanced}
\left\{ \e^{1/2} \|u_{\e}' - \realiz(\spike(\Phi^{FEM,\kappa,p}_{\e}))'\|_{L^2(I)} 
	+ \|u_{\e} - \realiz(\spike(\Phi^{FEM,\kappa,p}_{\e})) \|_{L^2(I)} \right\} 
\leq &\, C e^{-\beta p},
\\
\|u_{\e} - \realiz(\spike(\Phi^{FEM,\kappa,p}_{\e})) \|_{L^\infty(I)}
\leq &\, C e^{-\beta p},
\label{eq:spikinglinfty}
\end{align}
and $\realiz(\spike(\Phi^{FEM,\kappa,p}_{\e}))(\pm1) = 0$.

For a constant $\tilde{C} = \tilde{C}(\beta)>0$ depending only on $\beta$,
the SNN depth and size are bounded as follows: 
\begin{align}
\label{eq:spikingbldepthsize}
\depth\left(\spike(\Phi^{FEM,\kappa,p}_{\e})\right)
	\leq &\, \tilde{C} p (1+\log_2(p))
	,
	\qquad
\size\left(\spike(\Phi^{FEM,\kappa,p}_{\e})\right)
	\leq
		\tilde{C} p^2
. 
\end{align}
The weights in the hidden layers are independent of $u_{\e}$
and depend only on $\kappa$, $p$, $\e$ and $\beta$.
\end{theorem}
\begin{remark}\label{rmk:ReLu2Spk}
The presently used reasoning to infer expression rate bounds
for spiking NN architectures from rates proved for ReLU NNs
naturally also applies to other results, e.g. those in 
\cite{OS2023,JOdiss,MOPS22_2877} and also for 
so-called \emph{operator networks} of strict ReLU type in 
\cite{MS23_2948}.
\end{remark} 

\section{$\tanh$ Neural Network Approximations}
\label{sec:tanhanal}

In Section \ref{sec:model},
we have seen in Remark \ref{rem:blexp}
that when the reaction coefficient function $b(x)$ is constant,
then the boundary layer functions are known explicitly
and are given by \eqref{eq:explicitbl}.
Particularly simple NN approximations of these boundary layer functions
can be obtained with NNs which have one hidden layer and use as activation function
\begin{align*}
\tanh(x) = \tfrac{1-\exp(-2x)}{1+\exp(-2x)},
\qquad x\in\R
.
\end{align*}
This is the topic of Section \ref{sec:nntanhexp},
and is of independent interest.
In Section \ref{sec:nntanhsingpert},
we state the principal 
result of this section:
exponential DNN expression rate bounds in Sobolev norms 
on the set of solutions to \eqref{eq:de}--\eqref{eq:bc}
which are uniform in the singular perturbation parameter $\e\in (0,1]$.
Based on expression rate bounds
in Sections \ref{sec:nncalculus}--\ref{sec:nntanhidprod} 
and Appendix \ref{sec:nntanhunivarcheb},
in Section \ref{sec:nntanhanal} 
we construct deep $\tanh$-activated NN approximations of analytic functions,
and in particular of the smooth term in \eqref{eq:explicitbl},
sharpening previous results in \cite{DLM2021}.

In Section \ref{sec:PrfMainRslt},
we prove the main result of Section \ref{sec:tanhanal}
by combining the $\tanh$-NN approximation of the smooth term
with a $\tanh$-NN approximation of the boundary layer components of the solution
$u^\e$ developed in Section \ref{sec:nntanhexp}.

Throughout Section \ref{sec:tanhanal},
we will use the convention that
for a function $F\in W^{1,\infty}(D)$ for a domain $D\subset\R^d$, $d\in\N$,
the $W^{1,\infty}(D)$-norm is defined as
$\norm[W^{1,\infty}(D)]{ F } 
	= \max\{ \norm[L^\infty(D)]{ F }, 
		\max_{j=1}^d \norm[L^\infty(D)]{ \tfrac{\partial}{\partial x_j} F }\}$.
\subsection{$\tanh$ Emulation of the Exponential Function}
\label{sec:nntanhexp}

We analyze $\tanh$ NN approximations of the exponential function
in Lemma \ref{lem:tanhnnexp} below,
based on the following observation.

\begin{lemma}
\label{lem:tanhexp}
For all $x_0\geq 0$ and all $x\geq 0$ there holds
\begin{align*}
E(x) := &\, \snormc{ \exp(-x) - \exp(x_0) \tfrac12( 1- \tanh( \tfrac12 x + \tfrac12 x_0 ) ) }
	\leq \exp( -x_0 )
	,
	\\
\snorm{ E'(x) } := &\, \snormc{ \tfrac{\partial E}{\partial x}(x) }
	\leq 2 \exp( -x_0 )
.
\end{align*}
\end{lemma}

\begin{proof}
We start by noting that for all $x\in\R$
\begin{align*}
( 1 - \tanh( x ) )
	= &\, \tfrac{2\exp(-2x)}{1+\exp(-2x)}
	,
	\qquad
\tfrac12 ( 1 - \tanh( \tfrac12 x ) )
	= \tfrac{\exp(-x)}{1+\exp(-x)}
	,
	\\
\snormc{ \exp(-x) - \tfrac12 ( 1 - \tanh( \tfrac12 x ) ) }
	= &\, \snormc{ \exp(-x) - \tfrac{\exp(-x)}{1+\exp(-x)} }
	= \tfrac{1}{1+\exp(-x)} \exp(-2x)
.
\end{align*}
Using this result in the fourth step below, we obtain that
for all $x_0\geq 0$ and all $x\geq 0$
\begin{align*}
E(x) 	:= 
&\, \snormc{ \exp(-x) - \exp(x_0) \tfrac12 ( 1- \tanh( \tfrac12 x + \tfrac12 x_0 ) ) }
	\\
	= &\, \snormc{ \exp(-x) - \exp(x_0) \tfrac{\exp(-x-x_0)}{1+\exp(-x-x_0)} }
	\\
	= &\, \exp(x_0) \snormc{ \exp(-x-x_0) - \tfrac{\exp(-x-x_0)}{1+\exp(-x-x_0)} }
	\\
	= &\, \exp(x_0) \tfrac{1}{1+\exp(-x-x_0)} \exp(-2x-2x_0)
	\\
	= &\, \tfrac{1}{1+\exp(-x-x_0)} \exp(-2x-x_0)
	\\
	\leq &\, \exp(-2x-x_0) 
	\leq \exp(-x_0)
.
\end{align*}
In addition, we obtain 
\begin{align*}
E'(x) 
	= &\, \tfrac{1}{1+\exp(-x-x_0)} \cdot -2 \exp(-2x-x_0) 
	\\
	&\, + \tfrac{-1}{(1+\exp(-x-x_0) )^2} \cdot - \exp(-x-x_0) \exp(-2x-x_0)
	\\
	= &\, \Big( - \tfrac{2}{1+\exp(-x-x_0)} 
	+ \tfrac{1}{(1+\exp(-x-x_0) )^2} \exp(-x-x_0) \Big)
		\exp(-2x-x_0) 
	,
	\\
\snorm{ E'(x) }
	\leq &\, 2 \exp(-2x-x_0) 
	\leq 2 \exp(-x_0)
.
\end{align*}
We used that the absolute value of the negative term is larger than that of the positive term,
hence $\snorm{E'(x)}$ is bounded from above by the absolute value of the negative term.
\end{proof}

As a result, we have the following 
\emph{shallow $\tanh$ NN approximation rate bound
of the exponential function $\exp(-\cdot )$.}
This result is of independent interest, as 
exponential boundary layer functions appear
in a wide range of multivariate, singular perturbation 
problems (see e.g. \cite{chang2023singular} and the references there).
\begin{lemma}
\label{lem:tanhnnexp}
For all $\tau\in(0,1]$ there exists a $\tanh$ NN $\Phi^{\exp}_{\tau}$
such that for all $x\geq 0$ there holds
\begin{subequations}
\label{eq:tanhnnexperr}
\begin{align}
\label{eq:tanhnnexperri}
\snorm{ \exp(-x) - \realiz( \Phi^{\exp}_{\tau} )(x) }
	\leq &\, \exp(-\tau),
	\\
	\label{eq:tanhnnexperrii}
\snorm{ -\exp(-x) - \realiz( \Phi^{\exp}_{\tau} )'(x) }
	\leq &\, \exp(-\tau),
\end{align}
\end{subequations}
and such that 
$\depth( \Phi^{\exp}_{\tau} ) = 2$
and
$\size( \Phi^{\exp}_{\tau} ) = 4$.
\end{lemma}

\begin{proof}
We set $x_0 = \log(2/\tau)$, 
such that $\exp(x_0) = \tfrac{2}{\tau}$,
and define the $\tanh$ NN
\begin{align*}
\Phi^{\exp}_{\tau} 
	:= &\, \left( \left( \tfrac12, \tfrac12 x_0 \right), 
		\left( - \tfrac12 \exp(x_0), \tfrac12 \exp(x_0) \right) \right)
	\\
	= &\, \left( \left( \tfrac12, \tfrac12 \log(2/\tau) \right), 
		\left( - \tfrac1\tau, \tfrac1\tau \right) \right)
,
\end{align*}
which implies that
\begin{align*}
\realiz( \Phi^{\exp}_{\tau} )(x) 
	= \exp(x_0) \tfrac12 ( 1 - \tanh( \tfrac12 x + \tfrac12 x_0 ) )
,
\qquad
\text{ for all }x\geq0
.
\end{align*}
Using Lemma \ref{lem:tanhexp},
we obtain the error bounds
\begin{align*}
\snorm{ \exp(-x) - \realiz( \Phi^{\exp}_{\tau} )(x) }
	\leq &\, \exp(-x_0)
	= \tfrac{\tau}{2}
,
\qquad
\text{ for all }x\geq0
,
	\\
\snorm{ -\exp(-x) - \realiz( \Phi^{\exp}_{\tau} )'(x) }
	\leq &\, 2 \exp(-x_0)
	= \tau
,
\qquad
\text{ for all }x\geq0
.
\end{align*}
From the definition of $\Phi^{\exp}_{\tau}$, we observe that
$\depth( \Phi^{\exp}_{\tau} ) = 2$
and
$\size( \Phi^{\exp}_{\tau} ) = 4$.
\end{proof}

\begin{remark}
\label{rem:sigmoidbl}
We can apply the above analysis 
also to the sigmoid activation
$$
\sigma(x) = \tfrac{\exp(x)}{1 + \exp(x)} = \tfrac{1}{1 + \exp(-x)},
\qquad
x\in\R.
$$
We observe that for all $x\in\R$ holds
$\sigma(-x) = \tfrac{\exp(-x)}{1+\exp(-x)} = \tfrac12 ( 1 - \tanh( \tfrac12 x ) )$.
Thus, the $\sigma$-NN 
\[
\Phi^{\sigma,\exp}_{\tau} 
	:= \big( \big( {-1,- x_0} \big), \big( \exp(x_0), 0 \big) \big)
	\\
	= \big( \big({-1,-\log(2/\tau)} \big), \big( \tfrac{2}{\tau}, 0 \big) \big)
\]
satisfies 
\[
\depth( \Phi^{\sigma,\exp}_{\tau} ) = 2,\;\;
\size( \Phi^{\sigma,\exp}_{\tau} ) = 3,\;\;\mbox{and}\;\;
\realiz( \Phi^{\sigma,\exp}_{\tau} ) 
	= \exp(x_0) \sigma(-\cdot-x_0)
	= \realiz( \Phi^{\exp}_{\tau} ).
\]
Here $\Phi^{\exp}_{\tau}$ is the $\tanh$ NN from Lemma \ref{lem:tanhnnexp}.
In particular, $\Phi^{\sigma,\exp}_{\tau}$ 
satisfies \eqref{eq:tanhnnexperr}.
\end{remark}

\subsection{$\tanh$ NN Solution Approximation}
\label{sec:nntanhsingpert}
We 
state the main
result of this section, namely
\emph{$\e$-robust approximation rates 
for strict $\tanh$-activated deep NN approximations 
of solution families $\{ u^\e: 0<\e\leq 1\}\subset H^1_0(I)$
to the singularly perturbed, reaction-diffusion 
BVP \eqref{eq:de}--\eqref{eq:bc}.}
To leverage the Lemma~\ref{lem:tanhnnexp}, 
we consider again \eqref{eq:de}--\eqref{eq:bc} in 
the special case that 
the reaction coefficient $b(x)$ is constant, 
and equals $b\in\R$.
Without loss of generality, we assume that $b=1$. 
The solutions for general $b>0$ can be found
by solving the BVP with 
$\e^2$ replaced by $\e^2/b$ and 
$f$ replaced by $f/b$.
\begin{theorem}
\label{thm:nntanhsingpert}
Assume that (\ref{eq:analyticfb}) holds
and that the reaction coefficient function 
is constant and satisfies $b(x) = 1$ for all $x\in I = (-1,1)$. 
For $\e\in(0,1]$, 
let $u_{\e}$ be the solution of (\ref{eq:variational}).

Then, for all
$p\in\N$
there exists a $\tanh$ NN
$\Phi^{u_{\e},p}_{\e}$
such that, 
with positive constants $C$, $\beta$, 
independent of $\e\in (0,1]$ and of $p\geq 1$,
it holds that
$
\|u_{\e} - \realiz(\Phi^{u_{\e},p}_{\e}) \|_{W^{1,\infty}(I)}
\leq C e^{-\beta p}
$,
which implies that
\begin{align}
\label{eq:tanhbalanced}
\left\{ \e^{1/2} \|u_{\e}' - \realiz(\Phi^{u_{\e},p}_{\e})'\|_{L^2(I)} 
	+ \|u_{\e} - \realiz(\Phi^{u_{\e},p}_{\e}) \|_{L^2(I)} \right\} 
\leq &\, C e^{-\beta p},
\\
\|u_{\e} - \realiz(\Phi^{u_{\e},p}_{\e}) \|_{L^\infty(I)}
\leq &\, C e^{-\beta p}
.
\label{eq:tanhlinfty}
\end{align}

For a constant $\tilde{C} > 0$ independent of $f$, $p$, $\e$, $C$ and $\beta$,
the network depth and size are bounded as follows: 
\begin{align}
\label{eq:tanhbldepthsize}
\depth( \Phi^{u_{\e},p}_{\e} )
	= &\, \ceil{ \log_2(p) }+1
	,
	\qquad
\size( \Phi^{u_{\e},p}_{\e} )
	\leq
		\tilde{C} p
. 
\end{align}
The weights and biases in the hidden layers are independent of $u_{\e}$
and depend only on $p$, $\e$ and $\beta$.
\end{theorem}

A result corresponding to this theorem also holds for sigmoid-activated NNs.
We show this in Appendix \ref{sec:GenAct}.

The rest of this section is devoted to the proof of Theorem~\ref{thm:nntanhsingpert}.
In Section~\ref{sec:nncalculus}, we review results (in principle known) 
on a calculus of $\tanh$-activated deep NNs, in particular in 
Section~\ref{sec:nntanhidprod} $\tanh$-activated deep NN emulation
of the identity and of products of real numbers.
Section~\ref{sec:nntanhanal} addresses the $\tanh$-emulation of 
analytic functions, which are based on the novel 
emulations bounds of \Cheb polynomials by deep $\tanh$-activated
NNs, which are proved in Appendix~\ref{sec:nntanhunivarcheb}.
\subsection{Calculus of NNs}
\label{sec:nncalculus}
In the following sections, 
we will construct NNs from smaller networks
using a \emph{calculus of NNs},
which we now recall from \cite{PV2018}.
The results cited from \cite{PV2018} were derived for
NNs which only use the ReLU activation function, but they also hold
for networks with different activation functions without modification.
\begin{proposition}[Parallelization of NNs {{\cite[Definition 2.7]{PV2018}}}]
  \label{prop:parallel}
  For $d,L\in\N$ let $\Phi^1 = $ 
  $ \left
    ((A^{(1)}_1,b^{(1)}_1),\ldots,(A^{(1)}_L,b^{(1)}_L)\right
  ) $ and
  $\Phi^2 = \Big( (A^{(2)}_1,b^{(2)}_1),\ldots,$ \linebreak
  $(A^{(2)}_L,b^{(2)}_L)\Big) $ be two NNs with input dimension $d$ and depth $L$.  Let the
  \emph{parallelization} $\Parallel{\Phi^1,\Phi^2}$ of $\Phi^1$ and
  $\Phi^2$ be defined by
  \begin{align*}
    \Parallel{\Phi^1,\Phi^2} := &\, ((A_1,b_1),\ldots,(A_L,b_L)),
    &&
    \\
    A_1 = &\, \begin{pmatrix} A^{(1)}_1 \\ A^{(2)}_1 \end{pmatrix},
    \quad
    A_\ell = \begin{pmatrix} A^{(1)}_\ell & 0\\0&A^{(2)}_\ell \end{pmatrix},
    &&
       \text{ for } \ell = 2,\ldots L,
    \\
    b_\ell = &\, \begin{pmatrix} b^{(1)}_\ell \\ b^{(2)}_\ell \end{pmatrix},
    &&
       \text{ for } \ell = 1,\ldots L.
  \end{align*}

  Then,
  \begin{align*}
    \realiz(\Parallel{\Phi^1,\Phi^2}) (x)
    = &\, ( \realiz( \Phi^1 )(x), \realiz( \Phi^2 )(x) ),
        \quad
        \text{ for all } x\in\R^d,
    \\
    \depth(\Parallel{\Phi^1,\Phi^2}) = L, &
    \qquad
    \size(\Parallel{\Phi^1,\Phi^2}) = \size(\Phi^1) + \size(\Phi^2)
	.
  \end{align*}

\end{proposition}
The parallelization of more than two NNs is handled by repeated
application of Proposition \ref{prop:parallel}.

Similarly, we can also construct a NN 
emulating the sum of the realizations of two NNs.
\begin{proposition}[Sum of NNs]
  \label{prop:sum}

  For $d,N,L\in \N$ let
  $ \Phi^1 = \Big((A^{(1)}_1,b^{(1)}_1),\ldots,$ \linebreak
    $(A^{(1)}_L,b^{(1)}_L)\Big) $ and
  $\Phi^2 = \Big
  ((A^{(2)}_1,b^{(2)}_1),\ldots,(A^{(2)}_L,b^{(2)}_L)\Big)
  $
  be two NNs with input dimension $d$, output dimension $ N $ and
  depth $L$.  Let the \emph{sum} $\Phi^1+\Phi^2$ of $\Phi^1$ and
  $\Phi^2$ be defined by
  \begin{align*}
    \Phi^1 + \Phi^2 := &\, ((A_1,b_1),\ldots,(A_L,b_L)),
    &&
    \\
    A_1 =& \begin{pmatrix} A^{(1)}_1 \\ A^{(2)}_1 \end{pmatrix},
    \quad
    b_1 = \begin{pmatrix} b^{(1)}_1 \\ b^{(2)}_1 \end{pmatrix},
    \\
    A_\ell =& \begin{pmatrix} A^{(1)}_\ell & 0\\0&A^{(2)}_\ell \end{pmatrix},
                                                   \quad
                                                   b_\ell =  \begin{pmatrix} b^{(1)}_\ell \\ b^{(2)}_\ell \end{pmatrix},
                       &&
                          \text{ for } \ell = 2,\ldots L-1.
    \\
    A_L =& \begin{pmatrix} A^{(1)}_L & A^{(2)}_L \end{pmatrix},
                                       \quad
                                       b_L = b^{(1)}_L + b^{(2)}_L
.
  \end{align*}

  Then,
  \begin{align*}
    \realiz(\Phi^1+\Phi^2) (x)
    = &\, \realiz( \Phi^1 )(x) + \realiz( \Phi^2 )(x),
        \quad
        \text{ for all } x\in\R^d,
    \\
    \depth(\Phi^1+\Phi^2) = &\, L, \qquad
                              \size(\Phi^1+\Phi^2)
                              \leq \size(\Phi^1) + \size(\Phi^2)
                                                                .
  \end{align*}
\end{proposition}

We will sometimes use the parallelization of networks
which do not have the same inputs.
\begin{proposition}[Full parallelization of NNs {{\cite[Setting 5.2]{EGJS2021}}}] 
\label{prop:parallSep}
For $L \in \N$ let
$ \Phi^1 = \Big(
	(A^{(1)}_1,b^{(1)}_1),\ldots,(A^{(1)}_L,b^{(1)}_L)
    \Big)$ 
and
$\Phi^2 = \Big(
	(A^{(2)}_1,b^{(2)}_1),$ \linebreak $\ldots,(A^{(2)}_L,b^{(2)}_L)
	\Big)$
be two NNs with the same depth $L$,
with input dimensions $N^1_0=d_1$ and $N^2_0=d_2$, respectively. 

Then, the NN defined by
\begin{align*}
\FParallel{\Phi^1,\Phi^2} := &\, ((A_1,b_1),\ldots,(A_L,b_L)),
\end{align*}
\begin{align*}
    A_\ell = &\, \begin{pmatrix} A^{(1)}_\ell & 0\\0&A^{(2)}_\ell \end{pmatrix},
	\qquad
    b_\ell = \begin{pmatrix} b^{(1)}_\ell \\ b^{(2)}_\ell \end{pmatrix},
	\qquad
	\text{ for } \ell = 1,\ldots L
	,
\end{align*}
with $d = d_1+d_2$-dimensional input and depth $L$, 
called \emph{full parallelization of $\Phi^1$ and $\Phi^2$}, 
satisfies that 
for all $x = (x_1,x_2) \in \R^d$ with $x_i \in \R^{d_i}, i = 1,2$ 
\begin{align*} 
\realiz(\FParallel{\Phi^1,\Phi^2}) (x_1,x_2) 
	= &\, \left(\realiz(\Phi^1)(x_1), \realiz(\Phi^2)(x_2)\right)
\end{align*}
and
$\size(\FParallel{\Phi^1,\Phi^2}) 
	= \size(\Phi^1) + \size(\Phi^2)$.
\end{proposition}

Finally, we recall the concatenation of two NNs.
\begin{definition}[Concatenation of NNs {{\cite[Definition 2.2]{PV2018}}}]
\label{def:pvconc}
For $L^{(1)}, L^{(2)} \in\N$,
let 
$ \Phi^1 = $
$ \Big(
	(A^{(1)}_1,b^{(1)}_1),\ldots,$ $(A^{(1)}_{L^{(1)}},b^{(1)}_{L^{(1)}})
	\Big) $ and
$\Phi^2 = \Big(
	(A^{(2)}_1,b^{(2)}_1),$ \linebreak
	$\ldots,(A^{(2)}_{L^{(2)}},b^{(2)}_{L^{(2)}})
	\Big)$
be two NNs 
such that the input dimension of $\Phi^1$, which we will denote by $k$, 
equals the output dimension of $\Phi^2$.
Then, the \emph{concatenation} of $\Phi^1$ and $\Phi^2$ 
is the NN of depth $L := L^{(1)} + L^{(2)} -1$ defined as
\begin{align*}
    \Phi^1 \bullet \Phi^2 := &\, ((A_1,b_1),\ldots,(A_L,b_L)),
    \\
    (A_\ell,b_\ell) = &\, (A^{(2)}_\ell,b^{(2)}_\ell),
	\qquad
    \text{ for }\ell=1,\ldots,L^{(2)}-1,
    \\
    A_{L^{(2)}} = &\, A^{(1)}_1 A^{(2)}_{L^{(2)}},
    \qquad
    b_{L^{(2)}} = A^{(1)}_1 b^{(2)}_{L^{(2)}} + b^{(1)}_1,
    \\
    (A_\ell,b_\ell) 
    	= &\, (A^{(1)}_{\ell-L^{(2)}+1},b^{(1)}_{\ell-L^{(2)}+1}),
    \qquad
	\text{ for }\ell=L^{(2)}+1,\ldots,L^{(1)}+L^{(2)}-1.
\end{align*}
\end{definition}
It follows immediately from this definition that
$\realiz( \Phi^1 \bullet \Phi^2 ) = \realiz( \Phi^1 ) \circ \realiz( \Phi^2 )$.

\subsection{$\tanh$ Emulation of Identity and Products}
\label{sec:nntanhidprod}

Unlike ReLU NNs, $\tanh$-activated NNs can not represent the 
identity exactly. As
various constructions require identity maps, we provide a 
corresponding $\tanh$ NN emulation of the identity.
We also recall the $\tanh$ NN emulation of products from \cite{DLM2021}.

\begin{lemma}[See {{\cite[Lemma 3.1]{DLM2021} and \cite[Corollary 3.7]{DLM2021}}}]
\label{lem:tanhidprod}
For all $\tau,M>0$
and all $L\in\N$, $L\geq 2$
there exists a $\tanh$-activated NN $\idnna{1}{L}{\tau}{M}$
of depth $L$, 
with input dimension one and output dimension one,
such that 
\begin{align}
\label{eq:assid}
\normc[W^{1,\infty}((-M,M))]{ \Id_\R - \realiz( \idnna{1}{L}{\tau}{M} ) } \leq \tau 
.
\end{align}
There exists $C>0$ such that for all $L\in\N$, $L\geq 2$ 
there holds
$\size(\idnna{1}{L}{\tau}{M}) \leq C L$
for a constant $C$ independent of $\tau,M$ and $L$,
and also the layer dimensions of the hidden layers 
(denoted by $N_1,\ldots,N_{L-1}$ in the notation of Definition \ref{def:NeuralNetworks})
are at most $C$.

For all $\tau,M>0$ 
there exists a $\tanh$ NN
$\prodnna{2}{\tau}{M}$
of depth $2$,
with input dimension two and output dimension one
such that
\begin{align*}
\normc[W^{1,\infty}((-M,M)^2)]{ \prod_{i=1}^2 x_i - \realiz( \prodnna{2}{\tau}{M} )(x_1,x_2) } 
	\leq &\, \tau
.
\end{align*}
There exists $C>0$ such that 
$\size( \prodnna{2}{\tau}{M} ) \leq C$
for a constant $C$ independent of $\tau,M$.
\end{lemma}

\begin{proof}
We first prove the statements regarding identity networks.
Without loss of generality, assume that $\tau\leq 1$
(if $\tau>1$,
we can use the identity network defined below with $1$ instead of $\tau$).

By \cite[Lemma 3.1]{DLM2021} 
there exists a $\tanh$ NN $\idnna{1}{2}{\tau}{M}$ 
of depth $2$ 
with input dimension one and output dimension one
and with a fixed size independent of $\tau,M$,
such that \eqref{eq:assid} holds.

For $L>2$ we use Definition \ref{def:pvconc} to define
$\idnna{1}{L}{\tau}{M} := \idnna{1}{2}{\tau/3}{M+\tau/3} \bullet \idnna{1}{L-1}{\tau/3}{M}$.
In this proof, we will use the shorthand notation 
$\Phi^1 := \idnna{1}{2}{\tau/3}{M+\tau/3}$
and
$\Phi^2 := \idnna{1}{L-1}{\tau/3}{M}$.
We can estimate the error by
\begin{align*}
&\,
\norm[L^\infty((-M,M))]{ \Id_\R - \realiz( \Phi^1 \bullet \Phi^2 ) }
	\\
	\leq &\, \norm[L^\infty((-M,M))]{ \Id_\R - \realiz( \Phi^2 ) }
		+ \norm[L^\infty((-M,M))]{ ( \Id_\R - \realiz( \Phi^1 ) ) \circ \realiz( \Phi^2 ) }
	\\
	\leq &\, \tau/3 + \tau/3
	\leq \tau
	,
	\\
&\,
\norm[L^\infty((-M,M))]{ \Id_\R' - \realiz( \Phi^1 \bullet \Phi^2 )' }
	\\
	\leq &\, \norm[L^\infty((-M,M))]{ \Id_\R' - \realiz( \Phi^2 )' }
		+ \normc[L^\infty((-M,M))]{ \big( ( \Id_\R - \realiz( \Phi^1 ) ) \circ \realiz( \Phi^2 ) \big)' }
	\\
	\leq &\, \norm[L^\infty((-M,M))]{ \Id_\R' - \realiz( \Phi^2 )' }
	\\
	&\, + \norm[L^\infty((-M-\tau/3,M+\tau/3))]{ \Id_\R' - \realiz( \Phi^1 )' }
			\norm[L^\infty((-M,M))]{ \realiz( \Phi^2 )' }
	\\
	\leq &\, \tau/3 + \tau/3 (1+\tau/3) 
	\leq \tau/3 + 2\tau/3
	\leq \tau
	.
\end{align*}

In terms of the constant $C$, which is independent of $\tau,M>0$,
for which $\size( \idnna{1}{2}{\tau}{M} ) \leq C$
(its existence follows from \cite[Lemma 3.1]{DLM2021}), 
it follows that the number of neurons in each hidden layer of $\idnna{1}{L}{\tau}{M}$
(which are denoted by $N_1,\ldots,N_{L-1}$ in the notation of Definition \ref{def:NeuralNetworks}) 
are all at most $C$.\footnote{
If the number of neurons in a layer 
is larger than the number of nonzero weights and biases in that layer,
any neurons for which all associated weights and biases vanish, can be removed.
If in a layer all weights and biases vanish,
the network realizes a constant function,
which means that the network can be replaced
by a network of depth and size at most $1$ 
which exactly emulates the constant function.
These facts are proved in \cite[Lemma G.1]{PV2018}.}
Therefore, the number of nonzero weights in each layer is at most $C^2$
and the number of nonzero biases in each layer is at most $C$,
giving a total network size of at most $LC(C+1)$,
which is the desired bound (when we write $C$ instead of the constant $C(C+1)$).

The statements for product networks correspond to \cite[Corollary 3.7]{DLM2021}.
\end{proof}

Identity networks with multiple inputs are obtained as the full parallelization 
of identity networks with one input.
\begin{definition}
\label{def:idmultivar}
For all $d\in\N$,
we define
$\idnna{d}{L}{\tau}{M} := \FParallel{ \idnna{1}{L}{\tau}{M}, \ldots, \idnna{1}{L}{\tau}{M} }$
as the full parallelization of $d$ identity networks from Lemma \ref{lem:tanhidprod}.
\end{definition}

\subsection{$\tanh$ Emulation of Analytic Functions}
\label{sec:nntanhanal}

Exponential convergence of $\tanh$ NNs for the approximation of analytic functions
is obtained by combining the result from Appendix \ref{sec:nntanhunivarcheb}
with the following classical result on polynomial approximation.
See e.g. \cite[Section 12.4]{Davis1975}. 
We use the formulation from \cite[Lemma 9]{Melenk1997}.

\begin{lemma}
\label{lem:interpolanal}
Let $I = (-1,1)$ and assume that 
$u: I \to \R$ is analytic on $\overline{I} = [-1,1]$,
i.e. there exist $C_u,K_u>0$ such that
for all $n\in\N_0$ holds
\begin{align}
\label{eq:analforinterpol}
\norm[L^\infty(I)]{ u^{(n)} } \leq C_u K_u^n n!
.
\end{align}

Then, there exist constants $C,\beta>0$ 
such that 
for all $p\in\N$ there exists a polynomial $v\in\mathcal{P}_p$
such that 
$\norm[W^{1,\infty}(I)]{ u - v } \leq C e^{-\beta p}$.
\end{lemma}
For example, such a polynomial $v$ 
can be obtained from $u$ by polynomial interpolation 
in the Gau\ss--Lobatto points, see \cite[Lemma 11]{Melenk1997},
or in the Clenshaw--Curtis points,
see \cite[Theorem 8.2]{Trefethen2020} 
(the Clenshaw--Curtis points are introduced 
in \cite[Chapter 2]{Trefethen2020} 
and referred to as ``Chebyshev points'').

\begin{proposition}
\label{prop:nntanhanal}
Assume that $u$ satisfies \eqref{eq:analforinterpol}.

Then, there exist constants $C,\beta>0$ 
such that for all $p\in\N$
there exists a $\tanh$ NN $\Phi^{u,p}$
such that 
\[
\norm[W^{1,\infty}(I)]{ u - \realiz( \Phi^{u,p} ) } \leq C e^{-\beta p}
\]
with
\[
\depth( \Phi^{u,p} ) = \ceil{ \log_2(p) }+1, \;\;\mbox{and}\;\;
\size( \Phi^{u,p} ) \leq \tilde{C} p
\]
for a constant $\tilde{C}>0$ independent of $p$, $u$, $C$ and $\beta$.

The weights and biases in the hidden layers are independent of $u$.
The weights and biases in the output layer 
are linear combinations of the function values of $u$ in the Clenshaw--Curtis points.
\end{proposition}

\begin{proof}
For all $p\in\N$, 
let $v\in\mathcal{P}_p$  and $\beta>0$ be as given by Lemma \ref{lem:interpolanal},
let $\delta = \exp(-\beta p)$
and let 
$\Phi^{u,p} := \Phi^{v,p}_\delta$ 
be the network constructed in Corollary \ref{cor:chebsums}.
Then, for a constant $C>0$ independent of $p$,
\begin{align*}
\norm[W^{1,\infty}(I)]{ u - \realiz( \Phi^{u,p} ) }
	\leq &\, \norm[W^{1,\infty}(I)]{ u - v } + \norm[W^{1,\infty}(I)]{ v - \realiz( \Phi^{v,p}_\delta ) }
	\\
	\leq &\, C \exp( -\beta p) + \exp( -\beta p) \sum_{\ell = 1}^p \snorm{ v_\ell }
	\leq 
		C \exp( -\beta p)
	.
\end{align*}
In addition, we recall from Corollary \ref{cor:chebsums} that
$\depth( \Phi^{u,p} ) = \depth( \Phi^{v,p}_\delta ) = \ceil{ \log_2(p) }+1$
and
$\size( \Phi^{u,p} ) = \size( \Phi^{v,p}_\delta ) \leq \tilde{C} p$ 
for a constant $\tilde{C}>0$ independent of $p$, $u$, $C$ and $\beta$.

We observe from the proof of Corollary \ref{cor:chebsums} 
that the hidden layer weights are independent of $u$,
and that the output layer weights are linear combinations of 
the \Cheb coefficients of $v$.
As mentioned in the text after Lemma \ref{lem:interpolanal},
we may take $v$ to be the interpolant of $u$ in the Clenshaw--Curtis points,
as in \cite[Theorem 8.2]{Trefethen2020}.
Then, the \Cheb coefficients of $v$ 
can be computed from the function values of $v$ in the Clenshaw--Curtis points,
which equal those of $u$ ($v$ interpolates $u$ in those points).
\end{proof}

Proposition \ref{prop:nntanhanal} 
also holds for NNs with more general activation functions.
In Appendix \ref{sec:GenAct}, 
we prove that it holds for NNs 
whose activation function is in $C^2(U)\setminus\mathcal{P}_1$
for a nonempty, connected open subset $U\subset\R$.
\subsection{Proof of Theorem~\ref{thm:nntanhsingpert}}
\label{sec:PrfMainRslt}
\begin{proof}
As in the proof of \cite[Theorem 16]{Melenk1997}, 
which we recalled in Proposition \ref{prop:Melenk1997},
for $\kappa_0>0$ and $\kappa\in(0,\kappa_0)$ as in that proposition
we distinguish two cases.
In the first, \emph{asymptotic} case, 
the polynomial degree is so large with respect to $\e$
that $u_{\e}$ can be approximated directly, 
without treating the boundary layers separately.
In the second, \emph{pre-asymptotic} case,
we treat the boundary layers separately 
according to \eqref{eq:decompexplicit}--\eqref{eq:explicitbl}.

If $\kappa p \e \geq \tfrac12$, 
then the first part of the proof of \cite[Theorem 16]{Melenk1997}
shows that
for all $p\in\N$
there exists a polynomial\footnote{
In \cite{Melenk1997},
$v\in\mathcal{P}_p$ is taken to be the interpolant of $u_{\e}$ 
in the $p+1$ Gau\ss--Lobatto points.}
$v\in\mathcal{P}_p$
such that 
$\norm[W^{1,\infty}(I)]{ u - v } \leq C e^{-\beta p}$
for constants $C,\beta>0$ independent of $p$ and $\e$.
This polynomial can be approximated by a $\tanh$ NN by Corollary \ref{cor:chebsums}.
Defining $\delta = \exp(-\beta p)$ and $\Phi^{u_{\e},p}_{\e} := \Phi^{v,p}_{\delta}$,
there exists a constant $C>0$ independent of $p$ and $\e$ such that
\begin{align*}
\norm[W^{1,\infty}(I)]{ u - \realiz( \Phi^{u_{\e},p}_{\e} ) }
\leq &\, \norm[W^{1,\infty}(I)]{ u - v } + \norm[W^{1,\infty}(I)]{ v - \realiz( \Phi^{v,p}_\delta ) }
\\
\leq &\, C \exp( -\beta p) + \exp( -\beta p) \sum_{\ell = 1}^p \snorm{ v_\ell }
\leq 
	C \exp( -\beta p)
	,
\end{align*}
where $(v_\ell)_{\ell=0}^p$ are the \Cheb coefficients of $v$.
In addition, we recall from Corollary \ref{cor:chebsums} that
$\depth( \Phi^{u_{\e},p}_{\e} ) = \depth( \Phi^{v,p}_\delta ) = \ceil{ \log_2(p) }+1$
and
$\size( \Phi^{u_{\e},p}_{\e} ) = \size( \Phi^{v,p}_\delta ) \leq \tilde{C} p$ 
for a constant $\tilde{C}>0$ independent of $f$, $p$, $\e$, $C$ and $\beta$.
In Corollary \ref{cor:chebsums}, the hidden layer weights depend only on $p$.

If $\kappa p \e < \tfrac12$,
we use the decomposition \eqref{eq:decomp}
and separately approximate the $\e$-independent smooth part $u^S_{\e}$
and the boundary layer functions, 
which equal
$\tilde{u}^{\pm}_{\e}(x) 
 = C^{\pm} e^{- (1\mp x)/\e} $
by Remark \ref{rem:blexp},
for some constants $C^{\pm} > 0$ which are bounded independently of $\e$.

We approximate $u^S_{\e}$
by the $\tanh$ NN $\Phi^{u^S_{\e},p}$
from Proposition \ref{prop:nntanhanal}.
There exist constants $C,\beta>0$ independent of $p$ and $\e$, 
and $\tilde{C}>0$ independent of $f$, $p$, $\e$, $C$ and $\beta$,
such that it satisfies
$\norm[W^{1,\infty}(I)]{ u^S_{\e} - \realiz( \Phi^{u^S_{\e},p} ) } \leq C e^{-\beta p}$,
$\depth( \Phi^{u^S_{\e},p} ) = \ceil{ \log_2(p) }+1$
and
$\size( \Phi^{u^S_{\e},p} ) \leq \tilde{C} p$.

For the approximation of $u^{\pm}_{\e}$,
we use the approximation of the exponential function 
from Proposition \ref{lem:tanhnnexp},
the concatenation from 
Definition \ref{def:pvconc}
and 
the identity network from 
Lemma \ref{lem:tanhidprod}
to define
\begin{align}
\label{eq:tanhbldef}
\Phi^{\pm}_{\e}
	:= ((A^{\pm}_1,b^{\pm}_1,\Id_{\R})) 
	\bullet \Phi^{\exp}_\tau 
	\bullet \idnna{1}{L-1}{\delta}{M} 
	\bullet ((A^{\pm}_2,b^{\pm}_2,\Id_{\R}))
	,
\end{align}
for 
$A^{\pm}_1 = C^{\pm} \exp(1) \in\R^{1 \times 1}$,
$b^{\pm}_1 = 0 \in\R^1$,
$A^{\pm}_2 = \mp 1/\e \in \R^{1 \times 1}$,
$b^{\pm}_2 = 1/\e +1 \in \R^1$,
$\tau := \exp(-\beta p) \e$,
$\delta := \exp(-\beta p) \e$,
$M := 2/\e +1$
and
$L := \depth( \Phi^{u^S_{\e},p} ) = \ceil{ \log_2(p) }+1$.
 
For $p=1$, we have $L=1$ 
and omit ``$\bullet \idnna{1}{L-1}{\delta}{M}$'' from the definition
(because identity networks of depth $0$ have not been defined).
This will allow minor simplifications in the error bounds 
and the bounds on the network size for the case $p=1$,
which we will not consider explicitly.
 
Denoting by $P$ the affine transformation 
$P:x \mapsto A^{\pm}_2 x + b^{\pm}_2 = (1\mp x)/\e + 1$,
the networks $\Phi^{\pm}_{\e}$ realize 
\begin{align}
\label{eq:tanhblrealiz}
\realiz( \Phi^{\pm}_{\e} )
	= &\, C^{\pm} \exp(1) \realiz( \Phi^{\exp}_{\tau} ) 
		\circ \realiz( \idnna{1}{L-1}{\delta}{M} ) \circ P
.
\end{align}
We note that for $x\in[-1,1]$
holds
$P(x) = ( 1 \mp x ) /\e + 1 \in [1,1+2/\e]$,
and that $\delta<1$,
from which it follows that 
$\realiz( \idnna{1}{L-1}{\delta}{M} )\circ P(x) 
	\in [1-\delta,1+\delta+2/\e] \subset [0,2+2/\e]$.
It is necessary to add a positive number 
to the input (or output) of the identity network 
in order to guarantee that the input of $\Phi^{\exp}_\tau$ is nonnegative.
This is necessary in order to apply the error bound for $\Phi^{\exp}_\tau$ on $[0,\infty)$.
Specifically, 
we added $+1$ in the input layer, 
which is compensated for by the factor $\exp(1)$ in the output layer.
Also, we see that the inputs of the identity network
are indeed bounded in absolute value by $M = 2/\e+1$.
The error can be bounded as follows:
\begin{align*}
&\,
\norm[L^\infty(I)]{ \tilde{u}^{\pm}_\e - \realiz( \Phi^{\pm}_{\e} ) }
	\\
	= &\, C^{\pm} \exp(1) \norm[L^\infty(I)]{ \exp(-\cdot) \circ \Id_\R \circ P
		- \realiz( \Phi^{\exp}_{\tau} ) \circ \realiz( \idnna{1}{L-1}{\delta}{M} ) \circ P }
	\\
	\leq &\, C^{\pm} \exp(1) \normc[L^\infty(I)]{ \big( \exp(-\cdot) \circ \Id_\R
		- \exp(-\cdot) \circ \realiz( \idnna{1}{L-1}{\delta}{M} ) \big) \circ P }
	\\
	&\, + C^{\pm} \exp(1) \normc[L^\infty(I)]{ \big( \exp(-\cdot) - \realiz( \Phi^{\exp}_{\tau} ) \big) 
		\circ \realiz( \idnna{1}{L-1}{\delta}{M} ) \circ P }
	\\
	\leq &\, C^{\pm} \exp(1) \normc[L^\infty((0,2+2/\e))]{ -\exp(-\cdot) } 
		\normc[L^\infty((1,1+2/\e))]{ \Id_\R - \realiz( \idnna{1}{L-1}{\delta}{M} ) }
	\\
	&\, + C^{\pm} \exp(1) \normc[L^\infty((0,2+2/\e))]{ \exp(-\cdot) - \realiz( \Phi^{\exp}_{\tau} ) }
	\\
	\leq &\, C^{\pm} \exp(1) ( \delta + \tau ) 
	\leq C \exp(-\beta p)
	,
\end{align*}
for a constant $C$ independent of $\e$ and $p$.
To obtain bounds on the error in the derivative,
we first estimate
\begin{align*}
&\, 
\normc[L^\infty((1,1+2/\e))]{ \big( \exp(-\cdot) \circ \Id_\R
		- \realiz( \Phi^{\exp}_{\tau} ) \circ \realiz( \idnna{1}{L-1}{\delta}{M} ) \big)'}
	\\
	\leq &\, \normc[L^\infty((1,1+2/\e))]{ \big( (-\exp(-\cdot)) \circ \Id_\R
		- ( -\exp(-\cdot) ) \circ \realiz( \idnna{1}{L-1}{\delta}{M} ) \big) \cdot \Id_\R'}
	\\
	&\, + \normc[L^\infty((1,1+2/\e))]{ \big( ( -\exp(-\cdot) ) 
		\circ \realiz( \idnna{1}{L-1}{\delta}{M} ) \big) 
		\cdot \big( \Id_\R' - \realiz( \idnna{1}{L-1}{\delta}{M} ) ' \big) }
	\\
	&\, + \normc[L^\infty((1,1+2/\e))]{ \big( (-\exp(-\cdot) ) - \realiz( \Phi^{\exp}_{\tau} )' \big) 
		\circ \realiz( \idnna{1}{L-1}{\delta}{M} ) \cdot \realiz( \idnna{1}{L-1}{\delta}{M} )' }
	\\
	\leq &\, \normc[L^\infty((0,2+2/\e))]{ \exp(-\cdot) } 
		\normc[L^\infty((1,1+2/\e))]{ \Id_\R - \realiz( \idnna{1}{L-1}{\delta}{M} ) }
		\normc[L^\infty((1,1+2/\e))]{ \Id_\R' }
	\\
	&\, + \normc[L^\infty((0,2+2/\e))]{ -\exp(-\cdot) } 
			\normc[L^\infty((1,1+2/\e))]{ \Id_\R' - \realiz( \idnna{1}{L-1}{\delta}{M} )' }
	\\
	&\, 
		+ \normc[L^\infty((0,2+2/\e))]{ (-\exp(-\cdot) ) - \realiz( \Phi^{\exp}_{\tau} )' }
			\normc[L^\infty((1,1+2/\e))]{ \realiz( \idnna{1}{L-1}{\delta}{M} )' }
	\\
	\leq &\, \delta + \delta + \tau(1+\delta)
	\leq 2\delta + 2\tau
	,
\end{align*}
and get
\begin{align*}
&\,
\norm[L^\infty(I)]{ (\tilde{u}^{\pm}_\e)' - \realiz( \Phi^{\pm}_{\e} )' }
	\\
	= &\, C^{\pm} \exp(1) \normc[L^\infty(I)]{ \big( \exp(-\cdot) \circ \Id_\R
		- \realiz( \Phi^{\exp}_{\tau} ) \circ \realiz( \idnna{1}{L-1}{\delta}{M} ) \big)' \circ P \cdot P' }
	\\
	\leq &\, C^{\pm} \exp(1) \normc[L^\infty((1,1+2/\e))]{ \big( \exp(-\cdot) \circ \Id_\R
		- \realiz( \Phi^{\exp}_{\tau} ) \circ \realiz( \idnna{1}{L-1}{\delta}{M} ) \big)'}
		\norm[L^\infty(I)]{ P' }
	\\
	\leq &\, C^{\pm} \exp(1) ( 2 \delta + 2 \tau ) /\e
	\leq C \exp(-\beta p)
	,
\end{align*}
where $C>0$ again denotes a constant independent of $\e$ and $p$.
The depth indeed equals
\begin{align*}
\depth( \Phi^{\pm}_{\e} ) 
	= \depth( \Phi^{\exp}_\tau ) - 1 + \depth( \idnna{1}{L-1}{\delta}{M} ) 
	= 2 - 1 + (L-1) = L
.
\end{align*}
The NN size can be estimated
by arguments similar to those used in 
the proof of Lemma \ref{lem:tanhidprod},
as follows.
The first $L-2$ hidden layer dimensions of $\Phi^{\pm}_{\e}$ 
equal those of $\idnna{1}{L-1}{\delta}{M}$,
which are bounded by $C_*$ because of Lemma \ref{lem:tanhidprod}.
The dimension of the last hidden layer of $\Phi^{\pm}_{\e}$
equals the dimension of the one hidden layer of $\Phi^{\exp}_\tau$, 
which is $1$.
As a result, 
all layer dimensions of $\Phi^{\pm}_{\e}$ are bounded by $C_*$,
and the network size is bounded by 
$L C_* ( C_* + 1 ) \leq \tilde{C} ( \log_2(p) + 1 )$
for some constant $\tilde{C}>0$ 
independent of $f$, $p$, $\e$, $C$ and $\beta$.

The last term $u^R_{\e}$ in \eqref{eq:decompexplicit} is small and can be neglected.
As shown in the last step in the proof of \cite[Theorem 16]{Melenk1997},
there exist constants $C,\beta >0$, independent of $p$ and $\e$, such that 
\begin{align*}
\norm[W^{1,\infty}(I)]{ u^R_{\e} }
	\leq C \exp(-\beta p).
\end{align*}

Finally, we use Proposition \ref{prop:sum} to add the subnetworks
and define
\begin{align*}
\Phi^{u_{\e},p}_{\e} 
	:= \Phi^{u^S_{\e},p} + \Phi^{+}_{\e} + \Phi^{-}_{\e}
.
\end{align*}
It satisfies the error bound
\begin{align*}
&\,
\norm[W^{1,\infty}(I)]{ u_{\e} - \realiz( \Phi^{u_{\e},p}_{\e} ) }
	\\
	\leq &\, \norm[W^{1,\infty}(I)]{ u^S_{\e} - \realiz( \Phi^{u^S_{\e},p} ) }
		+ \norm[W^{1,\infty}(I)]{ \tilde{u}^{+}_\e - \realiz( \Phi^{+}_{\e} ) }
		+ \norm[W^{1,\infty}(I)]{ \tilde{u}^{-}_\e - \realiz( \Phi^{-}_{\e} ) }
	\\
	&\, + \norm[W^{1,\infty}(I)]{ u^R_{\e} }
	\\
	\leq &\, C e^{-\beta p}
.
\end{align*}
Moreover,
$\depth( \Phi^{u_{\e},p}_{\e} ) = \ceil{ \log_2(p) }+1$
and
\begin{align*}
\size( \Phi^{u_{\e},p}_{\e} )
	\leq &\, \size( \Phi^{u^S_{\e},p} ) + \size( \Phi^{+}_{\e} ) + \size( \Phi^{-}_{\e} )
	\\
	\leq &\, \tilde{C} p + \tilde{C} ( \log_2(p) + 1 ) + \tilde{C} ( \log_2(p) + 1 )
	\leq \tilde{C} p
	,
\end{align*}
for a constant $\tilde{C}>0$ independent of $f$, $p$, $\e$, $C$ and $\beta$.

The hidden layer weights and biases of $\Phi^{u^S_{\e},p}$ 
from Proposition \ref{prop:nntanhanal}
are independent of $u^S_{\e}$ and depend only on $p$.
The hidden layer weights and biases of $\Phi^{\pm}_{\e}$ 
only depend on $\e$, $p$ and $\beta$.
\end{proof}
\begin{remark}\label{rmk:ExBCs}
Exact boundary conditions can be imposed 
by slightly adjusting the constants $C^{\pm}$ 
in the formula for the boundary layers.
As the approximation error is exponentially small in $L^\infty(I)$,
also the necessary change in $C^{\pm}$ is exponentially small,
and so is the additional error in 
$\e^{-1/2} \norm[L^2(I)]{\circ}$, $\norm[L^\infty(I)]{\circ}$ 
and $\e^{1/2} \norm[H^1(I)]{\circ}$.
\end{remark}

\section{Conclusions and Generalizations}
\label{sec:Concl}
We summarize the principal findings of the present paper,
i.e., robust expression rate bounds for solutions $u^\e$ of
the model singular perturbation problem in Section~\ref{sec:model}
by several classes of DNNs, with either ReLU, $\tanh$ or sigmoid activation, 
or of spiking type.

For the model singular perturbation problem in Section~\ref{sec:model}
we established robust w.r. to $\e$ exponential expression rate bounds
for the solution $u_\e$ of \eqref{eq:de}--\eqref{eq:bc} 
by deep neural networks. 
We considered in detail several types and architectures
of DNNs, and the impact of architecture and activation on the
approximation rates.

The robust, exponential expression rate bounds proved in
Section~\ref{sec:relunn} for strict ReLU NNs
in Propositions~\ref{prop:relubalanced} and \ref{prop:reluexp}
implied, with a ReLU NN-to-spiking NN conversion algorithm from 
\cite{SWBCPG2022}, corresponding expression rate bounds also 
for so-called spiking NNs, in Section~\ref{sec:spiking}.

We also proved in a particular case of \eqref{eq:de}--\eqref{eq:bc} 
with exponential boundary
layer functions in Section~\ref{sec:tanhanal} 
that $\tanh( )$-activated NNs provide better
(still robust, exponential) expression rates.
In Section~\ref{sec:GenAct},
we show the same result for sigmoid NNs.

These results indicate that 
\emph{in order to
resolve multivariate exponential boundary layers, deep NNs 
with $\tanh$ or sigmoid activations afford expression rates
which are uniform w.r. to the length-scale parameter}.
With these activations in particular,
boundary layer resolution in deep NN approximations of PDE solutions
does not require ``augmentation'' of the DNN feature space with 
analytic boundary layers as proposed e.g. 
in the recent \cite{chang2023singular} and in the references there.
See e.g. \cite{AD2022} for a recent computational approach 
which does not rely on such augmentations.

In the proofs in Section~\ref{sec:tanhanal}, use was made 
of novel $\tanh$-NN emulation rate bounds for \Cheb polynomials, 
which are of independent interest, with their lengthy proofs
relegated to Appendix~\ref{sec:nntanhunivarcheb}.

\backmatter

\begin{appendices}

\section{$\tanh$ Emulation of Univariate \Cheb Polynomials}
\label{sec:nntanhunivarcheb}

A key step in the expression rate analysis is the 
analysis of $\tanh$-expression rates of \Cheb polynomials.
As these rates are, due to the wide use of \Cheb polynomials
in spectral methods and regression (e.g. \cite{OS2023,RauhtChS}), 
of independent interest, we provide a detailed analysis. 
We also recall that corresponding
results for strict ReLU NNs have been obtained in \cite{OS2023}.
In this appendix, we maintain all notations from the main text.

For $m\in\N$, $m\geq2$,
we construct in Definition \ref{def:univarcheb} below
a $\tanh$ NN $\Psi^{\ch,m}_\delta$
approximating all univariate \Cheb polynomials
of degree $1,\ldots,m$ with 
$W^{1,\infty}((-1,1))$-error at most $\delta\in(0,1)$.
We denote by $T_k$, $k\in\N_0$ the \Cheb polynomial of degree $k$ (of the first kind),
normalized such that $T_k(1) = 1$ for all $k\in\N_0$.
As in \cite[Proposition 7.2.2]{JOdiss}
(see also \cite[Appendix A]{HOS2022}),
the network has a binary tree structure,
similar to the one in \cite[Proposition 4.2]{OPS2020}.
It is based on the recursion
\begin{align*}
T_{m+n} = 2 T_{m} T_{n} - T_{\snorm{m-n}},
	\quad
T_0(x) = 1, 
	\quad
T_1(x) = x,
	\qquad
&\text{ for all } x\in\R 
\\
&\text{ and } m,n\in\N_0.
\end{align*}
This recursion is related to the addition rule for cosines
and was first used for NN construction in \cite{TLY2019chebnet}.

The construction uses 
the identity and product networks from Lemma \ref{lem:tanhidprod}.
In Lemma \ref{lem:tanhidprod}, 
the size of product networks and identity networks is independent of the desired accuracy.
This allows us to make minor simplifications 
to the construction that was used in the context of ReLU NNs 
in \cite{OPS2020,JOdiss}.

\begin{definition}
\label{def:univarcheb}
For all $\delta\in(0,1)$ we define 
\begin{align*}
\Psi^{\ch,2}_\delta 
	:= &\, ((A^{\ch,2}_{\delta,2},b^{\ch,2}_{\delta,2},\Id_{\R^2})) 
		\bullet \FParallel{ \idnna{1}{2}{\delta}{1}, \prodnna{2}{\delta/4}{1} }
	\bullet ((A^{\ch,2}_{\delta,1},b^{\ch,2}_{\delta,1},\Id_{\R^3}))
\end{align*}
where
$A^{\ch,2}_{\delta,2} := \operatorname{diag}(1,2) \in \R^{2\times2}$ is a diagonal matrix
with diagonal entries $1$ and $2$,
$b^{\ch,2}_{\delta,2} := (0,-1)^\top \in \R^2$,
$A^{\ch,2}_{\delta,1} := (1,1,1)^\top \in \R^{3\times1}$ and $b^{\ch,2}_{\delta,1} := 0 \in \R^3$.
Its realization is 
\[
\realiz( \Psi^{\ch,2}_\delta ) : \R \to \R^2: x \mapsto 
	\big( \realiz( \idnna{1}{2}{\delta}{1} )(x) , 2 \realiz( \prodnna{2}{\delta/4}{1} )(x,x) - 1 \big).
\]
To define $\Psi^{\ch,m}_\delta $,
for $m\in\N$ satisfying $m>2$,
let $\tilde{m} := \min\{ 2^k : 2^k \geq m, k\in\N \} < 2m$.
This means that there exists $k\in\N$ such that $\tilde{m} = 2^k$.
Then $2^{k-1} = \tilde{m}/2\geq 2$ 
implies that $k\geq 2$,
which implies that $\tilde{m}/2 = 2^{k-1}$ 
is an even number.

Let $\theta := \delta/(4m^2)$
and note that $4m^2 \geq 36$.
We use the following auxiliary matrices and vectors.
Let $A^{\ch,m}_{\delta,1} \in \R^{(2m-\tilde{m}/2)\times (\tilde{m}/2)}$
be defined by 
\begin{align*}
(A^{\ch,m}_{\delta,1})_{ij} 
	:= 
	&\, \begin{cases}
	1 & \text{ if } i = j \leq \tilde{m}/2,
	\\
	1 & \text{ if } i > \tilde{m}/2 \text{ and } j = \tilde{m}/4 + \ceil{ \tfrac{ i - 1 - \tilde{m}/2 }{4} },
	\\
	0 & \text{ else},
	\end{cases}
\end{align*}
and let $b^{\ch,m}_{\delta,1} := 0 \in \R^{2m-\tilde{m}/2}$.
In addition, let $A^{\ch,m}_{\delta,2} \in \R^{m\times m}$ and $b^{\ch,m}_{\delta,2} \in \R^m$
be defined by 
\begin{align*}
(A^{\ch,m}_{\delta,2})_{ij} 
	:= &\, \begin{cases}
	1 & \text{ if } i = j \leq \tilde{m}/2,
	\\
	2 & \text{ if } i = j > \tilde{m}/2,
	\\
	-1 & \text{ if } i > \tilde{m}/2 \text{ is odd and } j = 1,
	\\
	0 & \text{ else}.
	\end{cases}
	\\
(b^{\ch,m}_{\delta,2})_i 
	:= &\, \begin{cases}
	-1 & \text{ if } i > \tilde{m}/2 \text{ is even},
	\\
	0 & \text{ else}.
	\end{cases}
\end{align*}
Then, we recursively define 
\begin{align}
\nonumber
\Psi^{\ch,m}_\delta 
	:= &\, ((A^{\ch,m}_{\delta,2},b^{\ch,m}_{\delta,2},\Id_{\R^m}))
		\bullet \FParallel{ \idnna{\tilde{m}/2}{2}{\theta}{2}, \prodnna{2}{\theta}{2}, \ldots, \prodnna{2}{\theta}{2} }
	\\
	&\, \bullet ((A^{\ch,m}_{\delta,1},b^{\ch,m}_{\delta,1},\Id_{\R^{2m-\tilde{m}/2}}))
		\bullet \Psi^{\ch,\tilde{m}/2}_{\theta}
,
\label{eq:defpsichm}
\end{align}
where the full parallelization contains $m - \tilde{m}/2$ product networks.
In the remainder of this definition, 
we will abbreviate 
$\Psi^{m}_{\delta} := \Psi^{\ch,m}_{\delta}$
and
$\Psi^{\tilde{m}/2}_{\theta} := \Psi^{\ch,\tilde{m}/2}_{\theta}$.
By construction, 
the realization $\realiz( \Psi^{m}_{\delta} ) : \R \to \R^m$ satisfies
\begin{align*}
\realiz( \Psi^{m}_{\delta} )_j 
	= &\, \realiz( \idnna{1}{2}{\theta}{2} ) \circ \realiz( \Psi^{\tilde{m}/2}_{\theta} )_j,
	\qquad
	\text{ if } j \leq \tilde{m}/2,
	\\
\realiz( \Psi^{m}_{\delta} )_j 
	= &\, 2 \realiz( \prodnna{2}{\theta}{2} )(
	\realiz( \Psi^{\tilde{m}/2}_{\theta} )_{\floor{j/2}}, 
	\realiz( \Psi^{\tilde{m}/2}_{\theta} )_{\ceil{j/2}} )
	-1,
	\qquad
	\text{ if } j > \tilde{m}/2 \text{ is even},
	\\
\realiz( \Psi^{m}_{\delta} )_j 
	= &\, 2 \realiz( \prodnna{2}{\theta}{2} )(
	\realiz( \Psi^{\tilde{m}/2}_{\theta} )_{\floor{j/2}}, 
	\realiz( \Psi^{\tilde{m}/2}_{\theta} )_{\ceil{j/2}} )
	- \realiz( \idnna{1}{2}{\theta}{2} ) \circ \realiz( \Psi^{\tilde{m}/2}_{\theta} )_1,
	\\
	&\, \text{ if } j > \tilde{m}/2 \text{ is odd}.
\end{align*}
We used that by Definition \ref{def:idmultivar}
the subnetwork $\idnna{\tilde{m}/2}{2}{\theta}{2}$
is the full parallelization of $\tilde{m}/2$ networks $\idnna{1}{2}{\theta}{2}$,
thus $\realiz( \idnna{\tilde{m}/2}{2}{\theta}{2} )(x)_j = \realiz( \idnna{1}{2}{\theta}{2} )(x_j)$.
\end{definition}

Properties of the NNs in Definition \ref{def:univarcheb} are as follows.
\begin{proposition}
\label{prop:univarcheb}
For all $m\in\N$, $m\geq 2$ and for every $\delta\in(0,1)$
the NN $\Psi^{\ch,m}_\delta$ 
in Definition \ref{def:univarcheb} satisfies
\begin{align*}
\norm[W^{1,\infty}((-1,1))]{ T_k - \realiz( \Psi^{\ch,m}_\delta )_k } \leq \delta,
	\qquad\text{ for all } k = 1,\ldots,m,
\end{align*}
and, for some constant $C$ which is independent of $m$ and $\delta$,
\[
\depth( \Psi^{\ch,m}_\delta ) = \ceil{ \log_2(m) }+1\;,\;\;\mbox{and}\;\;
\size( \Psi^{\ch,m}_\delta ) \leq C m.
\]
\end{proposition}
\begin{proof}
This proof is by induction with respect to the number of hidden layers,
which equals $\ceil{ \log_2(m) }$.
In Step 1, we treat the case $m=2$, for which we will use one hidden layer.
Then, for $m>2$ and $\tilde{m} := \min\{ 2^k : 2^k \geq m, k\in\N \}$,
assuming that the result has been shown for $\tilde{m}/2 < m$,
we prove the statements for $m$, increasing the number of hidden layers by one.
In Step 2, we give the error estimates, 
and in Step 3, we analyze the network depth and size.

\textbf{Step 1.}
For $m=2$, we have 
\begin{align*}
\norm[W^{1,\infty}((-1,1))]{ T_1 - \realiz( \Psi^{\ch,2}_\delta )_1 } 
	= \norm[W^{1,\infty}((-1,1))]{ \Id_\R - \realiz( \idnna{1}{2}{\delta}{1} ) }
	\leq \delta
,
\end{align*}
as well as
\begin{align*}
\norm[L^\infty((-1,1))]{ T_2 - \realiz( \Psi^{\ch,2}_\delta )_2 } 
	= &\, \norm[L^\infty((-1,1))]{ (2x^2-1) 
	- \big( 2 \realiz( \prodnna{2}{\delta/4}{1} )(x,x) - 1 \big) }
	\\
	= &\, 2 \norm[L^\infty((-1,1))]{ x^2 - \realiz( \prodnna{2}{\delta/4}{1} )(x,x) }
	\leq 2 \delta / 4 
	\leq \delta
\end{align*}
and
\begin{align*}
&\,
\norm[L^\infty((-1,1))]{ T_2' - \realiz( \Psi^{\ch,2}_\delta )_2' } 
	\\
	= &\, \norm[L^\infty((-1,1))]{ (2x^2-1)' 
	- \big( 2 \realiz( \prodnna{2}{\delta/4}{1} )(x,x) - 1 \big)' }
	\\
	\leq &\, \norm[L^\infty((-1,1))]{ 4x 
		- 2 [D\realiz( \prodnna{2}{\delta/4}{1} )]_1(x,x) - 2 [D\realiz( \prodnna{2}{\delta/4}{1} )]_2(x,x) }
	\\
	\leq &\, 2 \delta / 4 + 2 \delta / 4 
	\leq \delta,
\end{align*}
where $[D\realiz( \prodnna{2}{\delta/4}{1} )]_j$ denotes the derivative with respect to the $j$-th argument.
For the depth we obtain
with repeated application of the formula for the depth from Definition \ref{def:pvconc}
that
$\depth( \Psi^{\ch,2}_\delta ) = 1 -1 +2 -1 +1 = 2$.
To estimate the network size,
let $C_*>0$ be as in Lemma \ref{lem:tanhidprod},
such that $\size( \idnna{1}{2}{\delta}{1} ) \leq 2C_*$ and $\size( \prodnna{2}{\delta/4}{1} ) \leq 2C_*$.
The number of neurons in the hidden layer of $\Psi^{\ch,2}_\delta$
equals that of $\FParallel{ \idnna{1}{2}{\delta}{1}, \prodnna{2}{\delta/4}{1} }$,
which is at most $2\cdot2C_*$.
Thus, in the notation of Definition \ref{def:NeuralNetworks} we have $N_0 = 1$, $N_1 \leq 4C_*$ and $N_2 = 2$,
which gives $\size( \Psi^{\ch,2}_\delta ) \leq N_1(N_0+1) + N_2(N_1+1) \leq 4C_* \cdot 2 + 2 (4C_*+1)$,
which is a constant independent of $\delta$.
Below, we will use that for $m=2$ 
the number of nonzero weights and biases in the last layer is bounded by $5 C_* m$.
This follows from $N_2(N_1+1) = 2 (4C_*+1) \leq 10 C_*$,
using that w.l.o.g. $C_*\geq1$ 
(recall that $C_*$ is an upper bound on the size of identity networks and product networks).
Similarly, 
$\size( \Psi^{\ch,2}_\delta ) 
	\leq 4C_* \cdot 2 + 2 (4C_*+1) 
	\leq 18 C_* 
	\leq 9 C_* m$.

\textbf{Step 2.}
Let now $m\in\N$, $m>2$.

\textbf{Step 2a.}
We first estimate the emulation error in the $L^\infty(-1,1)$-norm.
The arguments which we use closely follow the proof of \cite[Proposition 7.2.2]{JOdiss}.
We use the shorthand notation 
$\Psi^{m}_{\delta} := \Psi^{\ch,m}_{\delta}$
and
$\Psi^{\tilde{m}/2}_{\theta} := \Psi^{\ch,\tilde{m}/2}_{\theta}$
already used in Definition \ref{def:univarcheb}.
From $\norm[L^\infty((-1,1))]{ T_k } = 1$ and
$\norm[L^\infty((-1,1))]{ T_k - \realiz( \Psi^{\tilde{m}/2}_{\theta} )_k } \leq \theta$
we obtain that
$\norm[L^\infty((-1,1))]{ \realiz( \Psi^{\tilde{m}/2}_{\theta} )_k } \leq 1 + \theta \leq 2$,
which means that the inputs of the identity networks and product networks in \eqref{eq:defpsichm}
are indeed bounded in absolute value by $2$,
as is necessary in order to apply the error bounds from Lemma \ref{lem:tanhidprod}.
Also, we will use that each component of 
$\realiz( \idnna{\tilde{m}/2}{2}{\theta}{2} )$
equals 
$\realiz( \idnna{1}{2}{\theta}{2} )$,
see Definition \ref{def:idmultivar}.
To simplify the notation, 
for all $k=1,\ldots,\tilde{m}/2$ 
within this proof we will abbreviate $\widetilde{T}_k = \realiz( \Psi^{\tilde{m}/2}_{\theta} )_k$.
Now, for $j\leq \tilde{m}/2$,
\begin{align*}
\norm[L^\infty((-1,1))]{ T_j - \realiz( \Psi^{m}_{\delta} )_j }
	= &\, \norm[L^\infty((-1,1))]{ T_j - \realiz( \idnna{1}{2}{\theta}{2} ) \circ \widetilde{T}_j }
	\\
	\leq &\, \norm[L^\infty((-1,1))]{ T_j - \widetilde{T}_j }
		+ \norm[L^\infty((-1,1))]{ ( \Id_\R - \realiz( \idnna{1}{2}{\theta}{2} ) ) \circ \widetilde{T}_j }
	\\
	\leq &\, \norm[L^\infty((-1,1))]{ T_j - \widetilde{T}_j }
		+ \norm[L^\infty((-2,2))]{ \Id_\R - \realiz( \idnna{1}{2}{\theta}{2} ) }
	\\
	\leq &\, \theta + \theta
	= 2 \theta
	\leq \delta
.
\end{align*}
For even $j>\tilde{m}/2$, we note that the terms $-1$ in $T_j$ and $\realiz( \Psi^{m}_{\delta} )_j$ cancel,
and obtain
\begin{align*}
\norm[L^\infty((-1,1))]{ T_j - \realiz( \Psi^{m}_{\delta} )_j }
	\leq &\, \norm[L^\infty((-1,1))]{ 2 T_{\floor{j/2}} T_{\ceil{j/2}} 
			- 2 \widetilde{T}_{\floor{j/2}} \widetilde{T}_{\ceil{j/2}} }
	\\
	&\, + \norm[L^\infty((-1,1))]{ 2 \widetilde{T}_{\floor{j/2}} \widetilde{T}_{\ceil{j/2}} 
			- 2 \realiz( \prodnna{2}{\theta}{2} )(
			\widetilde{T}_{\floor{j/2}}, 
			\widetilde{T}_{\ceil{j/2}} ) }
	\\
	\leq &\, 2 \norm[L^\infty((-1,1))]{ T_{\floor{j/2}} - \widetilde{T}_{\floor{j/2}} }
			\norm[L^\infty((-1,1))]{ T_{\ceil{j/2}} }
	\\
	&\, + 2 \norm[L^\infty((-1,1))]{ \widetilde{T}_{\floor{j/2}} }
			\norm[L^\infty((-1,1))]{ T_{\ceil{j/2}} - \widetilde{T}_{\ceil{j/2}} }
	\\
	&\, + 2 \theta
	\\
	\leq &\, 2 \cdot \theta \cdot 1 + 2 \cdot 2 \cdot \theta + 2 \cdot \theta
	= 8 \theta
	\leq \delta
.
\end{align*}
For odd $j>\tilde{m}/2$ 
\begin{align*}
\norm[L^\infty((-1,1))]{ T_j - \realiz( \Psi^{m}_{\delta} )_j }
	\leq &\, \norm[L^\infty((-1,1))]{ T_1 - \realiz( \idnna{1}{2}{\theta}{2} ) \circ \widetilde{T}_1 }
	\\
	&\, + \norm[L^\infty((-1,1))]{ 2 T_{\floor{j/2}} T_{\ceil{j/2}} 
			- 2 \widetilde{T}_{\floor{j/2}} \widetilde{T}_{\ceil{j/2}} }
	\\
	&\, + \norm[L^\infty((-1,1))]{ 2 \widetilde{T}_{\floor{j/2}} \widetilde{T}_{\ceil{j/2}} 
			- 2 \realiz( \prodnna{2}{\theta}{2} )(
			\widetilde{T}_{\floor{j/2}}, 
			\widetilde{T}_{\ceil{j/2}} ) }
	\\
	\leq &\, 2 \theta + 8 \theta = 10 \theta
	\leq \delta
.
\end{align*}

\textbf{Step 2b.}
We use the same notation as in Step 2a.
To derive $W^{1,\infty}((-1,1))$-bounds, we first recall that
$\norm[L^\infty((-1,1))]{ T_k' } = k^2$ for all $k\in\N$.
The fact that $\norm[L^\infty((-1,1))]{ T_k' } \geq k^2$
follows from $T'_m = m U_{m-1}$, which is the \Cheb polynomial of the second kind of degree $m-1$,
and $U_{m-1}(1) = m$, which is \cite[Section 1.5.1]{Gautschi2004}.
The opposite inequality 
$\norm[L^\infty((-1,1))]{ T_k' } \leq k^2$
is Markov's inequality, which holds for all polynomials of degree at most $k$.
It follows that for all $k=1,\ldots,\tilde{m}/2$ 
there holds 
$\norm[L^\infty((-1,1))]{ \widetilde{T}_k' }
	\leq \norm[L^\infty((-1,1))]{ T_k' } + \norm[L^\infty((-1,1))]{ T_k' - \widetilde{T}_k' }
	\leq k^2 + \theta
	\leq k^2 + 1$.
Now, for $j\leq \tilde{m}/2$,
\begin{align*}
&\,
\norm[L^\infty((-1,1))]{ T_j' - \realiz( \Psi^{m}_{\delta} )_j' }
	\\
	\leq &\, \norm[L^\infty((-1,1))]{ T_j' - \widetilde{T}_j' }
		+ \normc[L^\infty((-1,1))]{ \big( ( \Id_\R - \realiz( \idnna{1}{2}{\theta}{2} ) ) \circ \widetilde{T}_j \big)' }
	\\
	\leq &\, \norm[L^\infty((-1,1))]{ T_j' - \widetilde{T}_j' }
		+ \norm[L^\infty((-2,2))]{ \Id_\R' - \realiz( \idnna{1}{2}{\theta}{2} )' }
		\norm[L^\infty((-1,1))]{ \widetilde{T}_j' }
	\\
	\leq &\, \theta + \theta \cdot  ( j^2+1 )
	= (j^2+2) \theta
	\leq 3j^2 \theta
	\leq \delta
.
\end{align*}
For even $j>\tilde{m}/2$, the terms $-1$ in $T_j$ and $\realiz( \Psi^{m}_{\delta} )_j$ cancel,
and we obtain
\begin{align*}
&\,
\norm[L^\infty((-1,1))]{ T_j' - \realiz( \Psi^{m}_{\delta} )_j' }
	\\
	\leq &\, \normc[L^\infty((-1,1))]{ 2 T_{\floor{j/2}}' T_{\ceil{j/2}} 
		- 2 [ D \realiz( \prodnna{2}{\theta}{2} ) ]_1(
			\widetilde{T}_{\floor{j/2}}, 
			\widetilde{T}_{\ceil{j/2}} ) \widetilde{T}_{\floor{j/2}}' }
	\\
	&\, + \normc[L^\infty((-1,1))]{ 2 T_{\floor{j/2}} T_{\ceil{j/2}}' 
		- 2 [ D \realiz( \prodnna{2}{\theta}{2} ) ]_2(
			\widetilde{T}_{\floor{j/2}}, 
			\widetilde{T}_{\ceil{j/2}} ) \widetilde{T}_{\ceil{j/2}}' }
	\\
	\leq &\, \normc[L^\infty((-1,1))]{ 2 T_{\floor{j/2}}'
		( T_{\ceil{j/2}} - \widetilde{T}_{\ceil{j/2}} ) }
	\\
	&\, + \normc[L^\infty((-1,1))]{ 2 T_{\floor{j/2}}' - \widetilde{T}_{\floor{j/2}}' ) 
		\widetilde{T}_{\ceil{j/2}} }
	\\
	&\, + \normc[L^\infty((-1,1))]{ 2 \big( \widetilde{T}_{\ceil{j/2}}
		- [ D \realiz( \prodnna{2}{\theta}{2} ) ]_1(
			\widetilde{T}_{\floor{j/2}}, 
			\widetilde{T}_{\ceil{j/2}} ) \big) \widetilde{T}_{\floor{j/2}}' }
	\\
	&\, + \normc[L^\infty((-1,1))]{ 2 ( T_{\floor{j/2}} - \widetilde{T}_{\floor{j/2}} ) T_{\ceil{j/2}}' }
	\\
	&\, + \normc[L^\infty((-1,1))]{ 2 \widetilde{T}_{\floor{j/2}}
		 	( T_{\ceil{j/2}}' - \widetilde{T}_{\ceil{j/2}}' ) }
	\\
	&\, + \normc[L^\infty((-1,1))]{ 2 \big( \widetilde{T}_{\floor{j/2}}
		- [ D \realiz( \prodnna{2}{\theta}{2} ) ]_2(
			\widetilde{T}_{\floor{j/2}}, 
			\widetilde{T}_{\ceil{j/2}} ) \big) \widetilde{T}_{\ceil{j/2}}' }
	\\
	\leq &\, 2 \floor{j/2}^2 \cdot \theta 
		+ 2 \theta \cdot 2 
		+ 2 \theta \cdot ( \floor{j/2}^2+1 )
		+ 2 \theta \cdot \ceil{j/2}^2
		+ 2 \cdot 2 \cdot \theta
		+ 2 \theta \cdot ( \ceil{j/2}^2+1 )
	\\
	= &\, ( 4 \floor{j/2}^2 + 4 \ceil{j/2}^2 + 12 ) \theta
	=
		( 4 ( \floor{j/2} + \ceil{j/2} )^2 - 8 \floor{j/2} \ceil{j/2} + 12 ) \theta
	\\
	= &\, ( 4 j^2 - 8 \floor{j/2} \ceil{j/2} + 12 ) \theta
	\\
	\leq &\, 4 j^2 \theta
	\leq \delta
,
\end{align*}
where $[ D \realiz( \prodnna{2}{\theta}{2} ) ]_\ell$ 
denotes the derivative with respect to the $\ell$-th argument.
We used that $j > \tilde{m}/2 \geq 2$ 
and thus 
$\floor{j/2} \geq 1$ and $\ceil{j/2} \geq 2$,
such that
$8 \floor{j/2} \ceil{j/2} \geq 16$.
For odd $j>\tilde{m}/2$
\begin{align*}
&\,
\norm[L^\infty((-1,1))]{ T_j' - \realiz( \Psi^{m}_{\delta} )_j' }
	\\
	\leq &\, \normc[L^\infty((-1,1))]{ T_1' - \big( \realiz( \idnna{1}{2}{\theta}{2} ) \circ \widetilde{T}_1 \big)' }
	\\
	&\, + \normc[L^\infty((-1,1))]{ 2 T_{\floor{j/2}}' T_{\ceil{j/2}} 
		- 2 [ D \realiz( \prodnna{2}{\theta}{2} ) ]_1(
			\widetilde{T}_{\floor{j/2}}, 
			\widetilde{T}_{\ceil{j/2}} ) \widetilde{T}_{\floor{j/2}}' }
	\\
	&\, + \normc[L^\infty((-1,1))]{ 2 T_{\floor{j/2}} T_{\ceil{j/2}}'
		- 2 [ D \realiz( \prodnna{2}{\theta}{2} ) ]_2(
			\widetilde{T}_{\floor{j/2}}, 
			\widetilde{T}_{\ceil{j/2}} ) \widetilde{T}_{\ceil{j/2}}' }
	\\
	\leq &\, ( 1^2 + 2 ) \theta + ( 4 j^2 - 8 \floor{j/2} \ceil{j/2} + 12 ) \theta
	= ( 4 j^2 - 8 \floor{j/2} \ceil{j/2} + 15 ) \theta
	\\
	\leq &\, 4 j^2 \theta
	\leq \delta
.
\end{align*}

\textbf{Step 3.} 
We again consider $m\in\N$, $m>2$.

\textbf{Step 3a.}
To determine the depth, we obtain 
by repeated use of the formula for the depth from Definition \ref{def:pvconc} 
that for all $m\in\N$, $m>2$ holds
$\depth( \Psi^{\ch,m}_\delta ) = 1 -1 +2 -1 +1 -1 +\depth( \Psi^{\ch,\tilde{m}/2}_\theta ) 
	= 1 +\depth( \Psi^{\ch,\tilde{m}/2}_\theta )$.
Together with $\depth( \Psi^{\ch,2}_\delta ) = 2$,
this implies that for all $m\in 2^\N$ we have
$\depth( \Psi^{\ch,m}_\delta ) = \log_2(m)+1$,
because $\tilde{m} = m$ for all $m \in 2^\N$.
For general $m\in\N$, $m\geq2$ we obtain
$\depth( \Psi^{\ch,m}_\delta ) = 1 + ( \log_2( \tilde{m}/2 ) +1 ) = \log_2(\tilde{m}) +1 = \ceil{ \log_2(m) }+1$.

\textbf{Step 3b.}
To estimate the network size,
again let $C_*>0$ be as in Lemma \ref{lem:tanhidprod},
such that $\size( \idnna{1}{2}{\theta}{2} ) \leq 2C_*$ and $\size( \prodnna{2}{\theta}{2} ) \leq 2C_*$.
By Definition \ref{def:idmultivar},
it holds that
\begin{align*}
\FParallel{ \idnna{\tilde{m}/2}{2}{\theta}{2}, \prodnna{2}{\theta}{2}, \ldots, \prodnna{2}{\theta}{2} }
	= \FParallel{ \idnna{1}{2}{\theta}{2}, \ldots, \idnna{1}{2}{\theta}{2}, 
		\prodnna{2}{\theta}{2}, \ldots, \prodnna{2}{\theta}{2} }
	,
\end{align*}
which has depth $2$
and contains $\tilde{m}/2$ identity networks and $m - \tilde{m}/2$ product networks.
This full parallelization is defined by repeated application of Proposition \ref{prop:parallSep},
from which we see that its weight matrices are block diagonal matrices.
The number of nonzero coefficients in each submatrix
is bounded from above by the size of the subnetwork of which it is part, which is at most $2C_*$.
In particular,
each row and each column contain at most $2C_*$ nonzero coefficients.
From the definition of $A^{\ch,m}_{\delta,2}$,
we see that $\norm[0]{ A^{\ch,m}_{\delta,2} } \leq 2m$.
Denoting the last layer weight matrix and bias vector of 
$\FParallel{ \idnna{\tilde{m}/2}{2}{\theta}{2}, \prodnna{2}{\theta}{2}, \ldots, \prodnna{2}{\theta}{2} }$
by $A_2$ and $b_2$, respectively,
those of 
$((A^{\ch,m}_{\delta,2},b^{\ch,m}_{\delta,2},\Id_{\R^m}))
	\bullet \FParallel{ \idnna{\tilde{m}/2}{2}{\theta}{2}, \prodnna{2}{\theta}{2}, \ldots, \prodnna{2}{\theta}{2} }$
equal $A^{\ch,m}_{\delta,2} A_2$ and $A^{\ch,m}_{\delta,2} b_2 + b^{\ch,m}_{\delta,2}$.
Because in the matrix multiplication $A^{\ch,m}_{\delta,2} A_2$
each element of $A^{\ch,m}_{\delta,2}$ gets multiplied with all elements of a row of $A_2$,
of which at most $2C_*$ are nonzero,
we obtain that 
$\norm[0]{ A^{\ch,m}_{\delta,2} A_2 } 
	\leq \norm[0]{ A^{\ch,m}_{\delta,2} } 2C_*
	\leq 2 m 2C_*
	= 4 C_* m$.
Also, 
$\norm[0]{ A^{\ch,m}_{\delta,2} b_2 + b^{\ch,m}_{\delta,2} }
	\leq m$ because this is a vector in $\R^m$,
so the total number of nonzero coefficients in the last layer is bounded by $( 4C_* +1 ) m \leq 5 C_* m$ 
(w.l.o.g. $C_*\geq1$).
Denoting the first layer weight matrix and bias vector of 
$\FParallel{ \idnna{\tilde{m}/2}{2}{\theta}{2}, \prodnna{2}{\theta}{2}, \ldots, \prodnna{2}{\theta}{2} }$
by $A_1$ and $b_1$, respectively,
recall that each row and each column of $A_1$ contain at most $2C_*$ nonzero coefficients.
From the definition of $A^{\ch,m}_{\delta,1}$,
we see that each column has at most $5$ nonzero coefficients,
namely one for which $i = j$ 
and at most four for which $i > \tilde{m}/2$ and $j = \tilde{m}/4 + \ceil{ \tfrac{ i - 1 - \tilde{m}/2 }{4} }$.
The $j$-th column of $A_1 A^{\ch,m}_{\delta,1}$ 
is the sum of columns of $A_1$ 
multiplied with elements of the $j$-th column of $A^{\ch,m}_{\delta,1}$.
Each column of $A_1$ contains at most $2C_*$ nonzero coefficients
and each column of $A^{\ch,m}_{\delta,1}$ at most $5$,
thus each column of $A_1 A^{\ch,m}_{\delta,1}$ at most $10C_*$.
Now, the weight matrix in the second to last layer of 
$\Psi^{m}_{\delta}$
is the product of $A_1 A^{\ch,m}_{\delta,1}$
with the weight matrix of the last layer of $\Psi^{\tilde{m}/2}_{\theta}$.
Each element of that matrix is multiplied with all coefficients in one column of $A_1 A^{\ch,m}_{\delta,1}$,
of which at most $10C_*$ are nonzero.
As shown above, the number of nonzero weights and biases 
in the last layer of $\Psi^{\tilde{m}/2}_{\theta}$ is at most $5 C_* \tilde{m}/2 \leq 5 C_* m$,
which means that the total number of nonzero weights
in the second to last layer of $\Psi^{m}_{\delta}$ is at most $50 C_*^2 m$.
The bias vector in that layer has $2m-\tilde{m}/2 \leq 2m$ entries,
thus the total number of nonzero weights and biases
in the second to last layer of $\Psi^{m}_{\delta}$ is at most $52 C_*^2 m$
(w.l.o.g. $C_*\geq1$).
All layers of $\Psi^{m}_{\delta}$ except the last two
are identical to those of $\Psi^{\tilde{m}/2}_{\theta}$.
Thus, we find that 
$\size( \Psi^{m}_{\delta} ) 
	\leq 5 C_* m + 52 C_*^2 m + \size( \Psi^{\tilde{m}/2}_{\theta} ) 
	\leq 57 C_*^2 m + \size( \Psi^{\tilde{m}/2}_{\theta} )$.
For all $m\in 2^\N$, it holds that $\tilde{m} = m$
and we find by induction that
$\size( \Psi^{m}_{\delta} ) \leq 114 C_*^2 m$, because 
$\size( \Psi^{m}_{\delta} ) 
	\leq 57 C_*^2 m + 114 C_*^2 (\tilde{m}/2) 
	= 114 C_*^2 m$.
For general $m\in\N$, $m>2$, we obtain
$\size( \Psi^{m}_{\delta} ) 
	\leq 57 C_*^2 m + 114 C_*^2 (\tilde{m}/2) 
	\leq 57 C_*^2 m + 114 C_*^2 m 
	= 171 C_*^2 m$
and we recall that for $m=2$ holds
$\size( \Psi^{2}_{\delta} ) 
	\leq 9 C_* m$.
\end{proof}

In \cite[Lemma 3.2]{DLM2021}, 
a shallow $\tanh$ network is constructed
which approximates all univariate monomials of degree $1,\ldots,m$
and has width bounded by $C m$ for some constant $C>0$.
This implies that the size 
of the NN constructed in \cite[Lemma 3.2]{DLM2021}
is $O(m^2)$.
It does not imply that 
the size of that network can be bounded by a constant times $m$. 
\begin{corollary}
\label{cor:chebsums}
For all $p\in\N$ and $v\in\mathcal{P}_p$,
let 
$v = \sum_{\ell=0}^p v_\ell T_\ell$ denote the \Cheb expansion of $v$.

Then, for all $\delta\in(0,1)$, 
there exists a $\tanh$ NN $\Phi^{v,p}_{\delta}$
which satisfies
\begin{align*}
\norm[W^{1,\infty}((-1,1))]{ v - \realiz( \Phi^{v,p}_{\delta} ) } 
	\leq \delta \sum_{\ell=1}^p \snorm{v_\ell},
\end{align*}
and, for some constant $C>0$ which is 
independent of $p$, $\delta$ and of $v$, 
\[
\depth( \Phi^{v,p}_{\delta} ) = \ceil{ \log_2(p) }+1\;,\;\; 
\size( \Phi^{v,p}_{\delta} ) \leq C p.
\]
The hidden layer weights and biases only depend on $p$ and $\delta$
and are independent of $v$.
Those in the output layer are linear combinations of $(v_\ell)_{\ell=0}^p$.
\end{corollary}

\begin{proof}
For $p=1$, $v$ is an affine function, 
which can be realized exactly by a NN of depth $1$ and size at most $2$.

For $p\geq2$, 
we use the NN $\Psi^{\ch,p}_{\delta}$ 
from Definition \ref{def:univarcheb} and Proposition \ref{prop:univarcheb}
and define
\begin{align*}
\Phi^{v,p}_{\delta} 
	:= &\, ((A^{v,p}_{\delta},b^{v,p}_{\delta},\Id_\R)) \bullet \Psi^{\ch,p}_{\delta}
,
\end{align*}
where $A^{v,p}_{\delta} = ( v_1,\ldots,v_p ) \in \R^{1 \times p}$ 
and $b^{v,p}_{\delta} = v_0 \in \R$.
Its realization satisfies
\begin{align*}
\realiz( \Phi^{v,p}_{\delta} )(x)
	= &\, v_0+ \sum_{\ell=1}^p v_\ell \realiz( \Psi^{\ch,p}_{\delta} )_\ell(x),
	\qquad x\in\R,
	\\
\norm[W^{1,\infty}(I)]{ v - \realiz( \Phi^{v,p}_{\delta} ) }
	\leq &\, \sum_{\ell=1}^p v_\ell 
		\norm[W^{1,\infty}(I)]{ T_\ell - \realiz( \Psi^{\ch,p}_{\delta} )_\ell }
	\leq
		\sum_{\ell=1}^p \snorm{v_\ell} \delta
.
\end{align*}
For the formula for the NN depth, we compute
\[
\depth( \Phi^{v,p}_{\delta} )
	= \depth( ((A^{v,p}_{\delta},b^{v,p}_{\delta},\Id_\R)) ) - 1 + \depth( \Psi^{\ch,p}_{\delta} )
	= \depth( \Psi^{\ch,p}_{\delta} )
	= \ceil{ \log_2(p) }+1
	.
\]
To estimate the NN size, 
we observe from the Definition \ref{def:pvconc} of concatenation
that all layers of $\Phi^{v,p}_{\delta}$ except for the last layer
equal those of $\Psi^{\ch,p}_{\delta}$.
Denoting the weights and biases in the 
last layer of $\Psi^{\ch,p}_{\delta}$
by $A^{\ch,p}_{\delta}$ and $b^{\ch,p}_{\delta}$, respectively,
those in the last layer of $\Phi^{v,p}_{\delta}$ are
$A := A^{v,p}_{\delta} A^{\ch,p}_{\delta}$ and 
$b := A^{v,p}_{\delta} b^{\ch,p}_{\delta} + b^{v,p}_{\delta}$,
respectively.
Denoting by $N$ the dimension of the 
second to last layer of $\Psi^{\ch,p}_{\delta}$, 
$A \in \R^{1 \times N}$,
and each element of this matrix 
is the matrix product of 
the matrix $A^{v,p}_{\delta} \in \R^{1 \times p}$
with a column of $A^{\ch,p}_{\delta}$.
Hence 
$\norm[0]{ A } \leq \norm[0]{ A^{\ch,p}_{\delta} }$.
In addition, $b\in\R^1$, thus $\norm[0]{ b } \leq 1$.
Finally, we obtain that 
\begin{align*}
\size( \Phi^{v,p}_{\delta} )
	\leq \size( \Psi^{\ch,p}_{\delta} ) + 1
	\leq C p + 1
	\leq C p
,
\end{align*}
for a constant $C>0$ independent of $p$, $\delta$ and $v$.
The statement on the NN weights 
follows directly from the definition of $\Phi^{v,p}_{\delta}$.
\end{proof}
 
\begin{remark}\label{rmk:Fast}
By \cite[Theorem 3.13]{Rivlin1974}, 
one can efficiently compute numerically the \Cheb coefficients $( v_\ell )_{\ell=0}^p$
using the inverse fast Fourier transform.
The sum of their absolute values grows at most algebraically with $p$
as we have the upper bound
$\sum_{\ell=2}^p \snorm{ v_\ell } \leq p^4 \norm[L^\infty(I)]{ v }$.
For more details, see \cite[Section 2]{OS2023}.
\end{remark}

\section{General Activation Functions}
\label{sec:GenAct}

The results in the present paper considered specifically
$\e$-uniform DNN emulation rates for the solution set 
$\{ u^\e: 0<\e\leq 1\}\subset H^1_0(I)$ of 
\eqref{eq:de}--\eqref{eq:bc} 
by strict ReLU, spiking, and $\tanh$-activated deep NNs.
Lemma~\ref{lem:tanhidprod}, the key result in the proofs
of expression rate bounds for $\tanh$-DNNs, 
holds more generally,
as we state in Lemma \ref{lem:prodnetsactiv} below.
Based on this, a result similar to 
Theorem \ref{thm:nntanhsingpert} follows 
also for more general activations,
such as the sigmoid introduced in Remark \ref{rem:sigmoidbl}.

We prepare the proof of the extension of 
Lemma~\ref{lem:tanhidprod} with two auxiliary lemmas.
Throughout this section, 
we will use the calculus of NNs from Section \ref{sec:nncalculus}.
Those results hold regardless of the used activation function.
\begin{lemma}
\label{lem:idnetc1}
For a nonempty, connected, open subset $U\subset\R$,
consider an activation function $\varrho \in C^1(U) \setminus \mathcal{P}_{0}$.

For all $M\geq 1$ and $\tau>0$,
there exists a depth $2$ $\varrho$-network $\Phi^{\Id,\varrho}_{\tau,M}$
satisfying 
$\realiz(\Phi^{\Id,\varrho}_{\tau,M})(0) = 0$ 
and
$\norm[W^{1,\infty}((-M,M))]{ \Id - \realiz(\Phi^{\Id,\varrho}_{\tau,M}) } \leq \tau$.
Its number of neurons and network size are bounded independently of $\tau$ and $M$.
The dimension of the hidden layer is $1$.
\end{lemma}

\begin{proof}
Let $t_0\in U$ be such that $\varrho'(t_0) \neq 0$.
There exists $\delta>0$ such that 
$[t_0 - \delta, t_0 + \delta] \subset U$
and 
$\max_{t\in [t_0 - \delta, t_0 + \delta]} \snorm{ 1 - \varrho'(t) / \varrho'(t_0) } 
	\leq \tau / M$.

Now, let 
$$
\Phi^{\Id,\varrho}_{\tau,M} 
:= \left(\left( \tfrac{\delta}{M},t_0 \right),
	\left( \tfrac{M}{\delta \varrho'(t_0)},-\tfrac{M \varrho(t_0)}{\delta \varrho'(t_0)} \right)\right),
$$
which has depth $2$, size $4$, hidden layer width $1$, 
and realization
$$
\realiz(\Phi^{\Id,\varrho}_{\tau,M})(x) 
	= \tfrac{M}{\delta \varrho'(t_0)} \left( \varrho( t_0 + \tfrac{\delta}{M} x ) - \varrho( t_0 ) \right)
,
\qquad
x\in[-M,M].
$$
It follows that $\realiz(\Phi^{\Id,\varrho}_{\tau,M})(0) = 0$ 
and that for all $x\in[-M,M]$ there holds
\begin{align*}
\snorm{ 1 - \realiz(\Phi^{\Id,\varrho}_{\tau,M})'(x) }
	= &\, \snorm{ 1 - \varrho( t_0 + \tfrac{\delta}{M} x ) / \varrho'(t_0) }
	\leq \tau / M \leq \tau
	,
	\\
\snorm{ x - \realiz(\Phi^{\Id,\varrho}_{\tau,M})(x) }
	\leq &\, M \tau / M = \tau
	.
\end{align*}
The latter estimate follows from integrating the former one,
using exactness in $0$ of the identity network.
\end{proof}

\begin{lemma}
\label{lem:squarenetc2}
For a nonempty, connected, open subset $U\subset\R$,
consider the function
$f: \R\to\R: x\mapsto x^2$
and an activation function $\varrho \in C^2(U) \setminus \mathcal{P}_{1}$.

For all $M\geq 1$ and $\tau>0$,
there exists a depth $2$ $\varrho$-network $\Phi^f_{\tau,M}$
satisfying 
$\realiz(\Phi^f_{\tau,M})(0) = 0$ 
and
$\norm[W^{1,\infty}((-M,M))]{ f - \realiz(\Phi^f_{\tau,M}) } \leq \tau$.
Its number of neurons and network size are bounded independently of $\tau$ and $M$.
\end{lemma}

\begin{proof}
Let $\Phi^{\Id,\varrho'}_{\tau/(4M),M}$ 
be the identity network from Lemma \ref{lem:idnetc1}
with activation function $\varrho'$.
Because its hidden layer dimension is $1$,
all its weight matrices are of dimension $1\times 1$ 
and its bias vectors are of dimension $1$.
We identify them with real numbers.
We write
$\Phi^{\Id,\varrho'}_{\tau/(4M),M}
	= ((a^{(1)}_1,b^{(1)}_1),(a^{(1)}_2,b^{(1)}_2))$,
such that
$\realiz( \Phi^{\Id,\varrho'}_{\tau/(4M),M} )(x) 
	= a^{(1)}_2 \varrho'( a^{(1)}_1 x + b^{(1)}_1 ) + b^{(1)}_2$
for all $x\in[-M,M]$.

The idea of this proof is to construct a $\varrho$-NN 
which approximates the antiderivative of $\realiz(\Phi^{\Id,\varrho'}_{\tau/(4M),M})$,
multiplied by $2$.
To approximate the antiderivative of the constant term, 
we use an identity network with activation function $\varrho$.
Let 
$\Phi^{\Id,\varrho}_{\tau',M}$
be the identity network from Lemma \ref{lem:idnetc1}
with accuracy $\tau' = \tau/(4 M b^{(1)}_2)$.
We write
$\Phi^{\Id,\varrho}_{\tau',M} = ((a^{(0)}_1,b^{(0)}_1),(a^{(0)}_2,b^{(0)}_2))$.

Now, define 
$\Phi^f_{\tau,M}
	:= ((A_1,b_1),(A_2,b_2))$,
where 
\begin{gather*}
A_1 = ( a^{(1)}_1, a^{(0)}_1 )^\top,
\quad
b_1 = ( b^{(1)}_1, b^{(0)}_1 )^\top,
\quad
A_2 = ( \tfrac{ 2 a^{(1)}_2 }{ a^{(1)}_1 }, 2 b^{(1)}_2 a^{(0)}_2 ),
\\
b_2 = 2 b^{(1)}_2 b^{(0)}_2
	-  \Big( \tfrac{ 2 a^{(1)}_2 }{ a^{(1)}_1 } \varrho( b^{(1)}_1 ) + 2 b^{(1)}_2 \left( a^{(0)}_2 \varrho( b^{(0)}_1 ) + b^{(0)}_2 \right) \Big)
,
\end{gather*}
so that for all $x\in[-M,M]$ there holds
\begin{align*}
\realiz( \Phi^f_{\tau,M} )(x)
	= &\, \tfrac{ 2 a^{(1)}_2 }{ a^{(1)}_1 } \varrho( a^{(1)}_1 x + b^{(1)}_1 ) 
		+ 2 b^{(1)}_2 \realiz( \Phi^{\Id,\varrho}_{\tau',M} )(x)
	\\
	&\, - \Big( \tfrac{ 2 a^{(1)}_2 }{ a^{(1)}_1 } \varrho( b^{(1)}_1 ) 
		+ 2 b^{(1)}_2 \realiz( \Phi^{\Id,\varrho}_{\tau',M} )(0) \Big)
.
\end{align*}
We see that $\realiz( \Phi^f_{\tau,M} )(0) = 0$
and that
\begin{align*}
\realiz( \Phi^f_{\tau,M} )'(x)
	= &\, 2 a^{(1)}_2 \varrho( a^{(1)}_1 x + b^{(1)}_1 ) 
		+ 2 b^{(1)}_2 \realiz( \Phi^{\Id,\varrho}_{\tau',M} )'(x)
	\\
	= &\, 2 \realiz(\Phi^{\Id,\varrho'}_{\tau/(4M),M})(x) 
		+ 2 b^{(1)}_2 ( \realiz( \Phi^{\Id,\varrho}_{\tau',M} )'(x) - 1 )
	,
	\\
\snorm{ 2 x - \realiz( \Phi^f_{\tau,M} )'(x) }
	\leq &\,    \snorm{ 2 x - 2 \realiz(\Phi^{\Id,\varrho'}_{\tau/(4M),M})(x) }
		+ 2 b^{(1)}_2 \snorm{ \realiz( \Phi^{\Id,\varrho}_{\tau',M} )'(x) - 1 }
	\\
	\leq &\, 2 \tfrac{\tau}{4M} + 2 b^{(1)}_2 \tau'
	= \tfrac{\tau}{2M} + \tfrac{\tau}{2M} 
	= \tau/M
.
\end{align*}
Integrating this error bound gives
$\snorm{ x^2 - \realiz( \Phi^f_{\tau,M} )(x) }
	\leq \tau$ 
for all $x\in[-M,M]$.
The network has depth $2$, 
one hidden layer comprising two neurons,
and size $7$, independently of $\tau$ and $M$.
\end{proof}

\begin{lemma}
\label{lem:prodnetsactiv}
Lemma \ref{lem:tanhidprod} also holds 
if $\tanh$ is replaced by any activation function $\varrho \in C^2(U)\setminus\mathcal{P}_1$ 
for a nonempty, connected, open subset $U\subset\R$.

In addition, it holds that $\realiz( \idnna{1}{L}{\tau}{M} )(0) = 0$
and that $x_1 x_2 = 0$ implies $\realiz( \prodnna{2}{\tau}{M} )(x_1,x_2) = 0$.
\end{lemma}

\begin{proof}
Throughout this proof, we assume that $M\geq 1$.
If $M<1$, then we consider the networks constructed for $M=1$,
which also satisfy the statements in the lemma for $M<1$.

For $L=2$, the network $\idnna{1}{2}{\tau}{M} := \Phi^{\Id,\varrho}_{\tau,M}$
from Lemma \ref{lem:idnetc1} satisfies all the desired properties.

The definition and the analysis of identity networks of depth $L>2$ 
are identical to those in the proof of Lemma \ref{lem:tanhidprod}.
From the definition 
$\idnna{1}{L}{\tau}{M} := \idnna{1}{2}{\tau/3}{M+\tau/3} \bullet \idnna{1}{L-1}{\tau/3}{M}$
we inductively obtain that
$\realiz(\idnna{1}{L}{\tau}{M})(0) = 0$.

To construct product networks,
for all $\tau>0$, $M\geq1$,
let $\Phi^f_{\tau,M}$ be the $\varrho$-NN from Lemma \ref{lem:squarenetc2}
approximating $f:[-M,M]\to\R:x\mapsto x^2$.

Now, 
let $A_1\in\R^{3\times 2}$ be such that 
for all $x = (x_1,x_2)\in\R^2$ there holds
$A_1 x = ( \tfrac12 x_1 + \tfrac12 x_2, \tfrac12 x_1, \tfrac12 x_2 )$,
let $b_1 := 0 \in\R^3$,
$A_2 := ( 2, -2, -2 ) \in \R^{1\times 3}$
and $b_2 := 0 \in \R$.
Then, the NN 
$\prodnna{2}{\tau}{M} 
	:= ((A_2,b_2)) \bullet \Parallel{ \Phi^f_{\tau/6,M}, \Phi^f_{\tau/6,M}, \Phi^f_{\tau/6,M} } 
	\bullet ((A_1,b_1))$
has realization
$\realiz( \prodnna{2}{\tau}{M} )( x_1,x_2 )
	= 2\realiz( \Phi^f_{\tau/6,M} )( \tfrac12 x_1 + \tfrac12 x_2 ) 
		- 2\realiz( \Phi^f_{\tau/6,M} )( \tfrac12 x_1 ) - 2\realiz( \Phi^f_{\tau/6,M} )( \tfrac12 x_2 )$
for all $(x_1,x_2) \in \R^2$.
With $\realiz( \Phi^f_{\tau/6,M} )(0) = 0$ we obtain that 
$\realiz( \prodnna{2}{\tau}{M} )( x_1,0 )
	= 0
	= \realiz( \prodnna{2}{\tau}{M} )( 0,x_2 )
$
for all $x_1,x_2\in[-M,M]$.
To estimate the $L^\infty((-M,M)^2)$ error,
for all $(x_1,x_2)\in[-M,M]^2$ there holds
\begin{align*}
\snorm{ x_1 x_2 - \realiz( \prodnna{2}{\tau}{M} )(x_1,x_2) }
	\leq &\, 2 \snorm{ ( \tfrac12 x_1 + \tfrac12 x_2 )^2
		- \realiz( \Phi^f_{\tau/6,M} )( \tfrac12 x_1 + \tfrac12 x_2 ) }
	\\
	&\, + 2 \snorm{ ( \tfrac12 x_1 )^2 - \realiz( \Phi^f_{\tau/6,M} )( \tfrac12 x_1 ) }
	\\
	&\, + 2 \snorm{ ( \tfrac12 x_2 )^2 - \realiz( \Phi^f_{\tau/6,M} )( \tfrac12 x_2 ) }
	\\
	\leq &\, 6 \tau/6 
	= \tau
.
\end{align*}
For the error in the derivative with respect to $x_1$, we find that
for all $(x_1,x_2)\in[-M,M]^2$ there holds
\begin{align*}
\snorm{ x_2 - \tfrac{\partial}{\partial x_1} \realiz( \prodnna{2}{\tau}{M} )(x_1,x_2) }
	\leq &\, \snormc{ ( x_1 + x_2 )
		- 2 \realiz( \Phi^f_{\tau/6,M} )' ( \tfrac12 x_1 + \tfrac12 x_2 ) \cdot \tfrac12 }
	\\
	&\, + \snormc{ x_1 - 2 \realiz( \Phi^f_{\tau/6,M} )' ( \tfrac12 x_1 ) \cdot \tfrac12 }
		+ 0
	\\
	\leq &\, 2 \tau/6 
	\leq \tau
,
\end{align*}
where the factors $\tfrac12$ are due to the chain rule.
The analogous bounds for differentiation with respect to $x_2$ also hold.

Because the number of neurons of $\Phi^f_{\tau/6,M}$ 
is bounded independently of $M$ and $\tau$,
so is that of $\prodnna{2}{\tau}{M}$.
Denoting by $N_1$ the number of neurons in the hidden layer of $\prodnna{2}{\tau}{M}$,
because the input dimension is $N_0 = 2$ 
and the output dimension is $N_2=1$,
the total number of nonzero weights and biases is at most 
$\sum_{\ell=1}^2 N_\ell (N_{\ell-1}+1) = 4N_1+1$,
which is independent of $M$ and $\tau$.
\end{proof}

\begin{remark}
\label{rem:sigmoidappx}
As a result of Lemma \ref{lem:prodnetsactiv}, 
all constructions and results in 
Section \ref{sec:nntanhidprod}, 
Appendix \ref{sec:nntanhunivarcheb} 
and Section \ref{sec:nntanhanal}
also hold for NNs 
whose activation function satisfies the conditions of Lemma \ref{lem:prodnetsactiv},
which includes ReLU$^2$ and the sigmoid defined in Remark \ref{rem:sigmoidbl}.
\end{remark}

With the approximation of 
the exponential boundary layer functions by sigmoid NNs in Remark \ref{rem:sigmoidbl}, 
the main result of Section \ref{sec:tanhanal} also holds for sigmoid NNs.

\begin{theorem}
\label{thm:nnsigmoidsingpert}
Theorem \ref{thm:nntanhsingpert} also holds for sigmoid NNs.
\end{theorem}

\end{appendices}



\end{document}